\documentclass[12pt]{amsart}
\usepackage{amsmath}
\usepackage{amssymb}
\usepackage{amsthm}
\usepackage{amscd}
\usepackage{tikz-cd}
\usepackage{amsfonts}
\usepackage{graphicx}%
\usepackage{fancyhdr}

\theoremstyle{plain} \numberwithin{equation}{section}
\newtheorem{theorem}{Theorem}[section]
\newtheorem*{thma}{Theorem A}
\newtheorem*{thmb}{Theorem B}

\newtheorem{lemma}[theorem]{Lemma}
\newtheorem{prop}[theorem]{Proposition}
\theoremstyle{definition}
\newtheorem{notation}[theorem]{Notation}
\newtheorem*{note}{Note}
\newtheorem{defi}[theorem]{Definition}

\newtheorem{rem}[theorem]{Remark}

\newtheorem{exam}[theorem]{Example}
\newtheorem{claim}[theorem]{Claim}
\newtheorem{obs}[theorem]{Observation}

\title{Strong Solidity of the $q$-Gaussian Algebras for All $-1 < q < 1$}
\author[Stephen Avsec]{Stephen Avsec}
\address{Department of Mathematics\\
University of Illinois, Urbana, IL, USA 61801}\email[Stephen Avsec]{savsec2@math.uiuc.edu}

\begin{document}
\begin{abstract}
The main result of this paper is to establish the weak* completely contractive approximation property (w*CCAP) for the q-Gaussian algebras for all values of $q \in [-1, 1]$ and any number of generators. We use this to establish that the q-Gaussian algebras are strongly solid in the sense of Popa and Ozawa for $q \in (-1, 1)$. 
\end{abstract}
\thanks{The author acknowledges support of the National Science Foundation through a fellowship funded by the grant ``EMSW21-MCTP: Research Experience for Graduate Students'' (NSF DMS 08-38434).}

\maketitle
\section{Introduction}
Approximation properties of C*-algebras and von Neumann algebras have provided a fundamental tool for many landmark results and applications of operator algebras. The strongest approximation property is the (weak*) completely positive approximation property which was shown by Effros and Choi in \cite{effroschoi} to be equivalent to nuclearity in the C*-algebra setting and shown by Connes in \cite{connesinj} to be equivalent to injectivity in the von Neumann algebra setting. These are equivalent to amenability in the group setting. A less restrictive property is the (weak*) completely bounded approximation property, which Haagerup first introduced in \cite{haagerup79} and \cite{deCannHaagerup}.
This property has become very important following the seminal work of Ozawa and Popa (\cite{OP1} and \cite{OP2}) and the recent follow up paper of Ozawa (\cite{Ozawa11}). It is equivalent to weak amenability of groups. 

The q-Gaussian variables were introduced by Bo\.{z}ejko and Speicher in \cite{BS1} as an interpolation between classical Gaussian variables in the case q = 1, fermionic variables in the case q = -1, and Voiculescu's free Gaussians in the case q = 0. These variables can be defined functorially from a real Hilbert space $H$ as being generated by self-adjoint elements $s_q(h)$, $h \in H$, which satisfy the moment formula
\[
\tau(s_q(h_1)\dots s_q(h_n)) = \left\{ \begin{array}{lr} 0 & \mbox{ if n is odd}\\
	\sum_{\rho \in P_2(n)} q^{\iota(\rho)} \prod_{\{j,k\}} \langle h_j, h_k \rangle & \mbox{ if n is even}
\end{array}. \right.
\]
Here $P_2(n)$ denotes the set of pair partitions on the set $\{1, \ldots, n\}$ and $\iota(\rho)$ denotes the number of crossings of the pair partition $\rho$. Notice that for $q = 0$, only the non-crossing partitions survive, and so we recover Voiculescu's free Gaussian algebras.

Later, the von Neumann algebras generated by the q-Gaussian variables, denoted by $\Gamma_q(H)$, were shown to have properties similar to those of the free group factors. Note that the free group factors $L\mathbb{F}_n$ are isomorphic to $\Gamma_0(H)$ where $n = \dim(H)$ by a famous result of Voiculescu in \cite{VDN}. Indeed, for a certain range of $q$ and $\dim(H)$, Bo\.{z}ejko, K\"{u}mmerer, and Speicher in \cite{BKS} established that the q-Gaussian algebras are factors. Ricard proved in \cite{ricard05} that $\Gamma_q(H)$ is a factor for all $-1 < q < 1$ and all $\dim(H) \geq 2$. That they do not have property $\Gamma$ was established by Sniady in \cite{sniady} for a certain range of $q$ and large dimension. Nou proved that the q-Gaussian algebras are non-injective for all $-1 < q < 1$ and $\dim(H) \geq 2$ in \cite{nou}.

Shlyakhtenko proved in \cite{shlyakh04} that the q-Gaussian algebras are solid in the sense of Ozawa for $|q| < \sqrt{2}-1$ using estimates on non-microstates free entropy dimension. In \cite{shlyakh09}, Shlyakhtenko further proved that the q-Gaussian algebras do have Cartan subalgebras for a small range of $q$. Recently, Dabrowski improved these estimates in \cite{dabrowski} to prove that an $n$-tuple of q-Gaussian variables have microstate free entropy dimension $n$ for $|q|n < 1$ and $q^2 n \leq 0.0169$. This result implies the present paper's results for this range of $q$ and $n$, though we use radically different techniques.

A weaker approximation property of C*-algebras is exactness, which only requires a C*-algebra to be a subalgebra of a nuclear C*-algebra. Thus, for a group, exactness implies that the group admits an amenable action on a compact Hausdorff space. Kennedy and Nica in \cite{kennedynica} proved that the C*-algebras generated by the q-creation operators are exact for all $-1 < q < 1$ and all dimensions. As a corollary, they prove that the C*-algebra generated by the q-Gaussian variables, $\mathcal{A}_q(H)$ is also exact since exactness is inherited by subalgebras. Since the CBAP implies exactness, we recover their result in this paper. 

Our first main result is that the q-Gaussian algebras have the (weak*) completely contractive approximation property. Specifically, the net of completely bounded maps which approximate the identity can be taken to be completely contractive.  
\begin{thma}
	For all $-1 < q < 1$ and all $\dim(H) \geq 2$, 
	\begin{enumerate}
			\item $\Gamma_q(H)$ has the weak* completely contractive approximation property.
	\item $\mathcal{A}_q(H)$ has the completely contractive approximation property.

	\end{enumerate}
	\label{CCAP}
\end{thma}

To prove this, we use Nou's result from \cite{nou}, plus a theorem which expresses Wick products in $\Gamma_q(H)$ as a linear combination of the products of two shorter Wick products using an inclusion/exclusion argument. These allow us to bound the cb-norm of the projections onto polynomials of degree $n$ by $C_q n^3$. Applying the q-Ornstein-Uhlenbeck semigroup provides the exponential decay which balances this polynomial growth.

Ozawa and Popa called a von Neumann algebra $\mathcal{M}$ \emph{strongly solid} if for all diffuse injective subalgebras $P \subset \mathcal{M}$, the normalizer $\mathcal{N}_\mathcal{M}(P):= \{u \in \mathcal{U}_\mathcal{M}| u^\ast P u = P\}$ generates an injective subalgebra. For a non-injective factor, this implies the absence of Cartan subalgebras. In \cite{OP1}, they proved that the free group factors and, in \cite{OP2} that the group von Neumann algebras of lattices in $SL_2(\mathbb{R})$ and $SL_2(\mathbb{C})$ are strongly solid. Ozawa and Popa's proof relied on these groups being weakly amenable and admitting a proper 1-cocycle associated to a representation which is weakly contained in the left regular representation. These conditions permit the von Neumann algebras to have the weak* completely bounded approximation property and to admit an $L^2$-compact deformation for which the associated bimodule is weakly contained in the coarse bimodule.

Following their work, Sinclair proved in \cite{sinclair} that the group von Neumann algebras of lattices in $SO(n,1)$ and $SU(n, 1)$ are strongly solid. These groups admit only a weakly-$\ell^2$ representation. These representations were defined by Shalom in \cite{shalom00} as one for which a sufficiently large tensor power was weakly contained in the left regular representation. Finally, Chifan and Sinclair in \cite{chifansinclair} proved that for a group which is weakly amenable and admits a \emph{quasi-1-cocycle} into a weakly-$\ell^2$ representation, the group von Neumann algebra is strongly solid. This class includes all hyperbolic groups and all lattices in rank one Lie groups.

The q-Gaussian algebras admit an ``s-malleable'' deformation in the sense of Popa (see, for example, \cite{popa_malleable_1} or \cite{popa_malleable_2}). However, we show that the bimodule associated to this deformation is not weakly contained in the coarse bimodule. The situation is similar in spirit to that of \cite{sinclair} but lacks the structure of an underlying group. In this paper, we find a subbimodule which is weakly contained in the coarse bimodule, following a discussion with Jesse Peterson regarding the connection between the associated bimodule and Schatten p-class operators. Combining this and an adjustment to Popa's estimate for s-malleable deformations (see Lemma 2.1 in \cite{popa08}), we were able to establish our second main result. 

\begin{thmb}
	For all $-1 < q < 1$ and all $\dim(H) < \infty$, $\Gamma_q(H)$ is strongly solid.
\end{thmb}

The paper is organized as follows. In the next section, we fix notation and give appropriate preliminary definitions and results. In Section 3, we prove the main result. In Section 4, we find an appropriate bimodule with which to prove that the q-Gaussians are strongly solid, and in Section 5, we complete the proof of strong solidity. 

\textbf{Acknowledgments:} The author is greatly indebted to Jesse Peterson for suggesting this strategy for finding the appropriate bimodule. The author would like to thank Yoann Dabrowski for providing him with an advanced copy of his preprint, Marius Junge for many useful discussions and careful reading, Pierre Fima for many enlightening discussions, and Austin Rochford for his help in editing.

\section{Preliminaries}
We use standard notation and results from von Neumann algebra theory (see e.g. \cite{takesakiI}) and operator space theory (see e.g. \cite{ER} or \cite{pisierOS}). 

\subsection{The CBAP}
	A von Neumann algebra $\mathcal{M}$ has the \emph{weak* completely bounded approximation property} (w*CBAP) if there exists a net of completely bounded, finite-rank maps $\varphi_\alpha: \mathcal{M} \to \mathcal{M}$ such that $\varphi_\alpha \to Id$ in the point-weak* topology and such that $\|\varphi_\alpha\|_{cb} \leq C$ for all $\alpha$. The minimal such constant is called the \emph{ Cowling-Haagerup constant} and is denoted by $\Lambda_{cb}(\mathcal{M})$. $\mathcal{M}$ has the \emph{w*CCAP} if $\Lambda_{cb}(\mathcal{M}) = 1$.

Cowling and Haagerup (\cite{haagerupcowling}) proved that for a discrete group $\Gamma$, $\Lambda_{cb}(\Gamma) = \Lambda_{cb}(L\Gamma)$. Since the free groups have Cowling-Haagerup constant 1 (\cite{haagerup79}), it was known that the free group factors had the w*CCAP. The equivalent definition for a C*-algebra would simply require that the net of finite-rank maps converge to the identity in the point-norm topology.

\subsection{Operators Spaces}
We use the row and column operator space structures on an abstract Hilbert space which are given by
\[
H_c = B(H, \mathbb{C}) \hspace{1cm} \mbox{and} \hspace{1cm} H_r = B(\mathbb{C}, \bar{H}) 
\] 
respectively. We have that $(H_c)^\ast = \bar{H}_r$ and $(H_r)^\ast = \bar{H_c}$. See Section 3.4 of \cite{ER} for detailed proofs.
We shall also utilize the Haagerup tensor product. Given two operator spaces $E$ and $F$, let $E \otimes_{} F$ be their algebraic tensor product. Let $x = (x_{ij}) \in \mathbb{M}_n(E \otimes_{} F)$. Define
\[
\|x\|_{h, n} = \inf_{r \geq 1}\{\|y\|_{\mathbb{M}_{n, r}(E)} \|z\|_{\mathbb{M}_{r, n}(F)} | x_{ij} = \sum_k y_{ik} z_{kj} \}
\]
The operator space tensor product defined by these norms is called the Haagerup tensor product and is denoted by $E \otimes_{h} F$. The following two facts about the Haagerup tensor product will be very useful.

\begin{rem}
	\label{hilbertop}
	For the row and column Hilbertian operator spaces, we have from Corollaries 5.8 and 5.10 of \cite{pisierOS} and Proposition 9.3.4 of \cite{ER}
\begin{enumerate}
	\item $H_c \otimes_{h} K_c = (H \otimes_{2} K)_c$\\
	\item $H_c \otimes_{h} \bar{K}_r \simeq \mathcal{K}(K, H)$\\
	\item $H_r \otimes_{h} \bar{K}_c \simeq S_1(K, H)$
\end{enumerate}
where the last two complete isometries are given by 
\[
(\xi \otimes_{h} \eta)(k) = \langle \eta, k \rangle \xi
\]
and
\[
(\xi \otimes_{h} \eta)(k) =\langle \eta, k \rangle \xi 
\]
where $\xi \in H$, $\eta \in K$, $T \in \mathcal{K}(K, H)$, and $S_{\eta,\xi}$ is the rank one map $S_{\eta, \xi}(\bar{k}) = \langle k, \eta \rangle \xi$.
\end{rem}

\begin{lemma}
	\label{mult}
	The multiplication map 
	\[
	m: L^2_r(\mathcal{M}) \otimes_{h} \overline{L^2_c(\mathcal{M})} \to L^1(\mathcal{M})
	\]
	defined by $m(a \otimes_{h} b) = a b^\ast$ is completely contractive for any finite von Neumann algebra $\mathcal{M}$. 
\end{lemma}

\begin{proof}
	We start by observing that $\overline{L_{}^{1}(\mathcal{M})}^\ast \simeq \mathcal{M}$ completely isometrically for any von Neumann algebra. See page 139 of \cite{pisierOS}. 

	Let 
	\[
	\pi: \mathcal{M}^{op} \to B(\overline{L^2(\mathcal{M})}) 
	\]
	be the usual right representation defined by $\pi(x)\hat{a} = \widehat{xa}$ where $\hat{a}$ denotes the image in $L_{}^{2}(\mathcal{M})$ of $a \in \mathcal{M}$. We claim that $\pi = m^\ast$. Let $x, a, b \in \mathcal{M}$. Let $\widehat{a} \in L_{r}^{2}(\mathcal{M})$ and $\hat{b} \in L_{c}^{2}(\mathcal{M})$ be the canonical images of $a$ and $b$ in the row and column spaces respectively, and let $\left\{ e_i \right\}$ be an orthonormal basis of $L_{}^{2}(\mathcal{M})$. We have
	\begin{align*}
		\langle \pi(x), \hat{a} \otimes_{h} \hat{b} \rangle &= Tr\left( \pi(x)^\ast \left( \hat{a} \otimes_{h} \hat{b} \right) \right)\\
		&= \sum_{i} \langle e_i, \pi(x)^\ast \left( \widehat{a} \otimes_{h} \hat{b} \right)e_i \rangle \\
		&= \sum_{i} \langle \pi(x)e_i, \langle \hat{b}, e_i \rangle \hat{a} \rangle \\
		&= \sum_{i} \langle  \pi(x)\langle e_i, \hat{b} \rangle e_i, \hat{a} \rangle \\
		&= \langle \pi(x)\widehat{b}, \hat{a} \rangle = \langle \widehat{xb}, \hat{a}\rangle = \tau(b^\ast x^\ast a) = \tau(x^\ast a b^\ast) = \langle x, m(\hat{a} \otimes_{h} \hat{b}) \rangle 
	\end{align*}
	Therefore, $m^\ast = \pi$ since $\mathcal{M}$ is dense in $L_{}^{2}(\mathcal{M})$ and so $\|m\|_{cb} = \|\pi\|_{cb} = 1$ since $\pi$ is a $\ast$-homomorphism.
\end{proof}

\begin{rem}
	Note that for every von Neumann algebra $\mathcal{M}$ there is a quotient map
	\[
	q: S_1(L_{}^{2}(\mathcal{M})) \to L_{}^{1}(\mathcal{M})
	\]
	which is the predual of the canonical representation
	\[
	\pi: \mathcal{M} \to B(L_{}^{2}(\mathcal{M})).
	\]
	The map $m$ above is simply this quotient map once we identify $S_1(L_{}^{2}(\mathcal{M}))$ with $L_{r}^{2}(\mathcal{M}) \otimes_{h} \overline{L_{c}^{2}(\mathcal{M})}$. 
\end{rem}

\subsection{The q-Fock Space and q-Gaussian Algebras}
Let $H$ be a real Hilbert space, $H_\mathbb{C} = H \otimes_\mathbb{R} \mathbb{C}$ its complexification, and $\mathcal{F}(H) = \oplus_{n \geq 0} H_\mathbb{C}^{\otimes n}$ be the algebraic Fock space over $H_\mathbb{C}$. Here $H_\mathbb{C}^{\otimes 0}$ is understood to be a one-dimensional space spanned by a unit vector $\Omega$. In \cite{BS1}, Bo\.{z}ejko and Speicher defined the following sesquilinear form on $\mathcal{F}(H)$.
\[
\langle h_1 \otimes \ldots \otimes h_n, k_1 \otimes \ldots \otimes k_m \rangle =\delta_{m, n} \sum_{\sigma \in S_n} q^{\iota(\sigma)} \prod_j \langle h_j, k_{\sigma(j)} \rangle 
\]
where $S_n$ denotes the symmetric group on $n$ characters, $\iota(\sigma)$ denotes the number of inversions of $\sigma \in S_n$, and $-1 \leq q \leq 1$. By the main result in \cite{BS1}, this form is nonnegative definite, in fact, strictly positive definite if $-1 < q < 1$ for each $n$ and thus defines an inner product. Denote by $\mathcal{F}_q(H)$ the Hilbert space completion of $\mathcal{F}(H)$ with respect to this inner product. Now for $h\in H$,$ h_1, \ldots, h_n \in H_\mathbb{C}$, define

\[
l_q(h)h_1 \otimes \ldots \otimes h_n = h \otimes h_1 \otimes \ldots \otimes h_n
\]
to be the left creation operator, and its adjoint the left annihilation operator

\[
l_q^\ast(h)h_1 \otimes \cdots \otimes h_n = \sum_{j=1}^n q^{j-1} \langle h, h_j \rangle h_1 \otimes \cdots \otimes \hat{h}_j \otimes \cdots \otimes h_n
\]
where $\hat{h}_j$ indicates that $h_j$ is omitted from the tensor. By \cite{BS1}, $l_q(h) \in B(F_q(H))$ for $-1 \leq q < 1$, and $l_q(h)$ is closable for $q = 1$. Let $s_q(h) = l_q(h) + l_q(h)^\ast$. We define the q-Gaussian von Neumann algebra to be 
\[
\Gamma_q(H) := \{s_q(h) | h \in H\}''
\]
for $-1 \leq q < 1$, and 
\[
\Gamma_1(H) = \{e^{is_1(h)} | h \in H\}''. 
\]
Furthermore, in \cite{BKS}, it is proved in Proposition 2.3 that the vector $\Omega$ is cyclic and separating, and defines a finite trace $\tau(x) = \langle \Omega, x \Omega \rangle $. Therefore, $\Gamma_q(H)$ is a finite von Neumann algebra in standard form. The following two results can also be found in \cite{BKS} (Proposition 2.7 and Theorem 2.11 respectively).

\begin{theorem}
	For each $\xi \in \mathcal{F}(H)$ there exists a unique element $W(\xi) \in \Gamma_q(H)$ such that $W(\xi) \Omega = \xi$. $W(\xi)$ is called the \emph{Wick product} of $\xi$.
	\label{wick}
\end{theorem}

\begin{theorem}
	Let $u : H \to K$ be a contractive map between real Hilbert spaces. There exists a trace-preserving unital completely positive (cput) map $\Gamma_q(u) : \Gamma_q(H) \to \Gamma_q(K)$ such that 
	\begin{enumerate}
		\item $\Gamma_q(u)$ is a *-automorphism if $u$ is an orthogonal tranformation.
		\item $\Gamma_q(u)$ is a *-embedding if $u$ is an inclusion.
		\item $\Gamma_q(u)$ is a conditional expectation if $u$ is a projection.
	\end{enumerate}
	\label{secquant}
\end{theorem}
In this sense, $\Gamma_q$ can be seen as a functor between the category of real Hilbert spaces with contractions and the category of $II_1$ factors with completely positive maps (see \cite{VDN} for the free case). 

\subsection{Wick products}
The following two maps were introduced in \cite{BKS} and studied in a very general setting in \cite{krolak}. Define a map 
\[
U_{n, k} : H_c^{\otimes n-k} \otimes_h \bar{H}_r^{\otimes k} \to B(\mathcal{F}_q(H)) 
\]
by 
\begin{align*}
	&U_{n,k} \left( (h_1 \otimes \ldots \otimes h_{n-k}) \otimes_h (\bar{h}_{n-k+1} \otimes \ldots \otimes \bar{h}_n)  \right)
	&= l_q(h_1) \ldots l_q(h_{n-k}) l_q^\ast(\bar{h}_{n-k+1}) \ldots l_q^\ast(\bar{h}_n)
\end{align*}
	and 
	\[
	R_{n,k}: \bar{H}_r^{\otimes_{}n-k} \otimes_{h} H_c^{\otimes_{}k} \to H^{\otimes_{}n} 	\]
	by
	\[
	R_{n,k}(h_1 \otimes_{} \cdots \otimes_{} h_{n-k} \otimes_{h} h_{n-k+1} \otimes_{} \cdots \otimes_{} h_n) = h_1 \otimes_{}\cdots \otimes_{} h_n
	\]
	It is shown in \cite{nou} that
	\begin{align*}
		&R_{n,k}^\ast \left( h_1 \otimes \ldots \otimes h_n  \right)\\
		&=\sum_{x \in S_n/S_{n-k} \times S_k} q^{\iota(x)} (h_{\sigma_x(1)} \otimes \ldots \otimes h_{\sigma_x(n-k)}) \otimes_h (h_{\sigma_x(n-k+1)} \otimes \ldots \otimes h_{\sigma_x(n)})
	\end{align*}
		where $x \in S_n/S_{n-k} \times S_k$ are right cosets, $\sigma_x \in x$ is the representative of $x$ with the fewest inversions, and $\iota(x) = \iota(\sigma_x)$.
		The following theorem is Theorem 1 in \cite{krolak} 
\begin{theorem}
	\label{wick}
	Let $\xi \in H^{\otimes n}$. 
	\[
	W(\xi) = \sum_{k = 0}^n U_{n, k} R_{n, k}^\ast(\xi) 	\]
for $\xi = h_1 \otimes \ldots \otimes h_n$.
\end{theorem}

\begin{obs}
	\label{rnk}
	We may associate to any right coset $x \in S_n/S_{n-k} \times S_k$ a subset $A \subset \{1, \ldots, n\}$ such that $|A| = k$ in the following way. For any permutation $\sigma \in x$, $\sigma(j) \in A^c$ for $1 \leq j \leq n-k$ and $\sigma(j) \in A$ for $n-k + 1 \leq j \leq n$. Suppose $A^c = (\alpha_1, \ldots, \alpha_{n-k})$ and $A = (\beta_1, \ldots, \beta_k)$. Then 
	\[
	\sigma_x(1, \ldots, n) = (\alpha_1, \ldots, \alpha_{n-k}, \beta_1, \ldots, \beta_k)
	\]
	From now on, we may replace a \emph{right coset} $x \in S_n/S_{n-k} \times S_k$ with its corresponding \emph{subset} of cardinality $k$ where convenient.
\end{obs}

\subsection{The q-Ornstein-Uhlenbeck Semigroup and its Markov Dilation}
Let $u_t: H \to H$ be the map $h \mapsto e^{-t} h$ for $t \geq 0$. $u_t$ is clearly a contraction, and so by Theorem \ref{secquant}, we have a trace-preserving, completely positive, unital (\emph{cput}) map $T_t = \Gamma_q(u_t)$. Since $u_s \circ u_t = u_{s + t}$, $T_s \circ T_t = T_{s + t}$ by functorality, and $T_0 = Id$. Therefore, $T_t$ is a cput semigroup. We shall denote by $N$ its (positive) generator, which is called the number operator. 

\begin{defi}
	We say a cput semigroup $T_t$ on a (semi-)finite von Neumann algebra $\mathcal{M}$ admits a \emph{Markov dilation} if there is a larger (semi-)finite von Neumann algebra $\widetilde{\mathcal{M}} $ with increasing filtration $ \widetilde{\mathcal{M}}= \vee_{t \geq 0} \widetilde{\mathcal{M}}_t$  $\widetilde{\mathcal{M}}_t \subset \widetilde{\mathcal{M}}_s$ when $t < s$ together with a sequence of $\ast$-homomorphisms $\varphi_t: \mathcal{M} \to \widetilde{\mathcal{M}}_t$ such that $E_s \circ \varphi_t(x) = \varphi_s \circ T_{t-s}(x)$ for all $t > s$, $x \in \mathcal{M}$, where $E_s$ denotes the conditional expectation onto $\vee_{s \geq t \geq 0} \mathcal{M}_t$.
\end{defi}
In this case, our Markov dilation is special in that $\tilde{\mathcal{M}}_t = \tilde{\mathcal{M}}$ for all $t > 0$ and the $\varphi_t$ are of the form $\alpha_t \circ \pi$ where $\alpha_t$ is an automorphism group of $\widetilde{\mathcal{M}}$, $\pi$ is a canonical inclusion of $\mathcal{M} \subset \widetilde{\mathcal{M}}$. In particular, let $\widetilde{\mathcal{M}} = \Gamma_q(H \oplus H)$ and let $R_t: H \oplus H \to H \oplus H$ be the rotation
\[
R_t =
\left( \begin{array}{cc}
	e^{-t}Id & -\sqrt{1 - e^{-2t}}Id\\
	\sqrt{1 - e^{-2t}}Id & e^{-t}Id
\end{array}\right)
\]
Let $\alpha_t = \Gamma_q(R_t)$. By Theorem \ref{secquant}, this can be extended to a group of *-automorphisms of $\Gamma_q(H \oplus H)$. Let $P_1: H \oplus H \to H$ be the projection onto the first coordinate and $E_1 = \Gamma_q(P_1)$. Then $T_t(x) = E_1 \circ \alpha_t(x)$ and $\Gamma_q(H) \subset \Gamma_q(H \oplus H)$ by including $H$ in the first coordinate. This identity is not difficult to show. 


\begin{defi}
	A von Neumann subalgebra $P \subset \mathcal{M}$ is rigid with respect to a continuous family of completely positive maps $\theta_t:\mathcal{M} \to \mathcal{M}$ if $\theta_t \to Id$ as $t \to 0$ uniformly on the unit ball of $L_{}^{2}(P)$.
\end{defi}

We have the following theorem follows immediately from Theorem 5.4 of \cite{popa06} regarding rigid subalgebras of $\Gamma_q(H)$ when $\dim(H) < \infty$. 
\begin{theorem}
	Let $B \subset \Gamma_q(H)$ be a von Neumann subalgebra. Then TFAE
	\begin{enumerate}
		\item $B$ is rigid with respect to $\alpha_t$.
		\item $B$ is rigid with respect to $T_t$.
		\item $B$ is atomic.
	\end{enumerate}
	\label{rigid}
\end{theorem}

\begin{proof}
3) $\Rightarrow$ 2):
$T_t$ is compact on $L_{}^{2}(\Gamma_q(H))\simeq \mathcal{F}_q(H)$ since $H^{\otimes n}$ is a finite dimensional eigenspace of $T_t$ with eigenvalue $e^{-nt}$. Let $B \subset \Gamma_q(H)$ be a diffuse subalgebra. Suppose $B$ is rigid with respect to $T_t$. Then there exists $t_0$ such that 
\[
\|T_t(u) - u\|_2 \leq \frac{1}{2}
\]
for all $t < t_0$ and $u \in \mathcal{U}(B)$. Let $A \subset B$ be a maximal abelian subalgebra of $B$. $A$ is diffuse since $B$ is diffuse, so there exists a unitary $v \in A$ such that $\tau(v^m) = 0$. The sequence $\hat{v}_m \in L_{}^{2}(\mathcal{M})$ converges weakly to $0$. Since $T_t$ is compact, we get that $\|T_t(\hat{v}_m)\|_2 \to 0$ for any fixed $t$. Therefore
\[
\|T_t(\hat{v}_m) - \hat{v}_m\|_2 = \|\hat{v}_m\|_2 = 1
\]
contradicting that $\|T_t(u) - u\|_2 \leq \frac{1}{2}$ for all $t < t_0$ and $u \in \mathcal{U}(B)$.

2) $\Rightarrow$ 3): Suppose $B$ is Type I. Since $\Gamma_q(H)$ is finite, $B = \oplus_{\alpha \in I} M_{n_\alpha}$ for some countable index set $I$ since $\Gamma_q(H)$ has a separable predual. We have projections $e_{\alpha}$ such that $e_\alpha B e_\alpha = M_{n_\alpha}$ and $\sum_{\alpha \in I} \tau(e_\alpha) = 1$. For any $\varepsilon > 0$, there is a finite set $F \subset I$ such that $\sum_{\alpha \in F} \tau(e_\alpha) > 1-\varepsilon$. Let $\iota_F: B \to L_{}^{2}(B)$ be the map $x \mapsto \sum_{\alpha \in F} e_\alpha x e_\alpha$. A simple estimate shows that $\|\iota_{F^c}\| \leq \sqrt{\varepsilon}$. Therefore 
\[
\lim_{t \to 0} \sup_{\|x\|_{\infty} \leq 1} \|T_t(\iota_F(x))\|_{2} = 0
\]
and so $T_t$ converges uniformly on $(B)_1$.

1) $\Rightarrow$ 2): For any $x \in L_{}^{2}(\Gamma_q(H))$ we have
\begin{align*}
	\|\alpha_t(x) - x \|_2^2 &= 2\langle x, x \rangle - \langle \alpha_t(x), x \rangle - \langle x, \alpha_t(x) \rangle \\
	&= 2(\langle x, x \rangle  - \langle x, T_t(x) \rangle )\\
	&\leq 2\|x\| \|T_t(x) - x\|_2\\
	\end{align*}
	and so if $T_t$ converges uniformly, $\alpha_t$ converges uniformly.
	
2) $\Rightarrow$ 1):
Similarly,
\begin{align*}
	\|T_t(x) - x\|_2 &= \|E_{\Gamma_q(H)} \left(  \alpha_t(x) - x \right)\|_2\\
	&\leq \|\alpha_t(x) - x\|_2
\end{align*}
so if $\alpha_t$ converges uniformly, $T_t$ converges uniformly. 
\end{proof}

\subsection{Central Limit Theorem}
From now on, we shall drop the subscript $q$ and simply assume that $q$ is a fixed parameter between -1 and 1. 
We shall need the following two results. The first is Proposition 2 in \cite{BS1}. 

\begin{theorem}
	Let $h_1, \ldots, h_{2n}$ be vectors in $H$. Then
	\[
	\tau(s(h_1) \ldots s(h_{2n})) = \sum_{\sigma \in P_2(2n)} q^{\iota(\sigma)} \prod_{\{i,j\} \in \sigma} \langle h_i, h_j \rangle 
	\]
	where $P_2(2n)$ denotes the pair partitions of the set $\{1, \ldots, 2n\}$ and $\iota(\sigma)$ denotes the number of crossings of the partition $\sigma$. 
	\label{moment}
\end{theorem}

\begin{theorem}
	Let $s_j(h) = s(h \otimes_{} e_j)$ for some orthonormal basis $\{e_j\}_j \subset \ell_{N}^{2}(\mathbb{R})$ and $h \in H$ for an arbitrary Hilbert space $H$. Consider the operator
\[
u_N(h) = N^{-\frac{1}{2}} \sum_{j = 1}^N s_j(h).
\]

Then 
\[
\tau(u_N(h_1)\cdots u_N(h_m)) = \tau(s(h_1)\cdots s(h_m)).
\] for all $h_1, \ldots, h_m \in H$. \label{CLT} \end{theorem}

\begin{proof}
	Observe that the map $\iota_N : H \to H \otimes \ell_{\mathbb{R}}^{2}(N)$ where
	\[
	\iota_N(h) = N^{-\frac{1}{2}} \sum_{j=1}^N h \otimes e_j
	\]
	is an isometric embedding. We apply Theorem \ref{secquant} and the theorem follows immediately.
\end{proof}

Fix a free ultrafilter $\mathcal{U}$ on the natural numbers. For a sequence of Banach spaces $\{X_n\}$, we may define the ultraproduct $X_\mathcal{U}$ by 
\[
X_\mathcal{U} = \prod\nolimits_\mathcal{U} X_n  := \prod\nolimits_{n} X_n/ I_\mathcal{U}
\]
where
\[
I_\mathcal{U} = \{(x_n) : \lim_{n, \mathcal{U}} \|x_n\|_{X_n} = 0\}.
\]
Define
\[
u_\mathcal{U}(h) = (u_N(h))^\bullet \in \prod\nolimits_{\mathcal{U}} L_{}^{p}(\Gamma_q(\ell_{N}^{2}(\mathbb{R}) \otimes H))
\]
where the notation $(x_N)^\bullet$ is used for the equivalence class of the sequence $(x_N)$ in the von Neumann algebra 
\[
\tilde{\mathcal{N}}_\mathcal{U} := \left( \prod_{N, \mathcal{U}} \overline{L_{}^{1}(\Gamma_q(H \otimes_{} \ell_{N}^{2}(\mathbb{R})))} \right)^\ast. 
\]
However, $\tilde{\mathcal{N}}_\mathcal{U}$ is not in general finite, but it contains a canonical finite subalgebra $\mathcal{N}_\mathcal{U}$ which is obtained as the image of bounded sequences in the Hilbert space $\prod_{\mathcal{U}}L^2\left( \Gamma_q(H \otimes \ell_{N}^{2}(\mathbb{R})) \right)$, obtained by the GNS construction for the trace
\[
\tau_{\mathcal{U}} = \lim_{N, \mathcal{U}}\tau_N,
\]
where $\tau_N$ is the vacuum trace associated to $\Gamma_q(H \otimes_{} \ell_{N}^{2}(\mathbb{R}))$. Thus there is a canonical inclusion
\[
L_{}^{p}(\mathcal{N}_\mathcal{U}) \subseteq \prod\nolimits_\mathcal{U}L_{}^{p}(\Gamma_q(H \otimes \ell_{N}^{2}(\mathbb{R}))).
\]
By Theorem \ref{CLT}, we have an injective *-homomorphism  $\pi_\mathcal{U}: \Gamma_q(H) \to \mathcal{N}_\mathcal{U} $ such that  $\pi_\mathcal{U}(s(h)) =  (u_N(h))^\bullet$. See \cite{pisierOS} Section 9.10 or \cite{raynaud2002} for more information regarding ultraproducts of finite von Neumann algebras. 

Let $u: H \to K$ be any contraction on real Hilbert spaces. Note that 
\[
u \otimes Id_N: H \otimes_{} \ell_{N}^{2}(\mathbb{R}) t\ K \otimes_{} \ell_{N}^{2}(\mathbb{R}) 
\]
is also a contraction. By Theorem \ref{secquant}, there is a completely positive map
\[
\Gamma_q(u \otimes_{} Id_N): \Gamma_q(H \otimes_{} \ell_{N}^{2}(\mathbb{R})) \to \Gamma_q(K \otimes_{} \ell_{N}^{2}(\mathbb{R})).
\]
Therefore we may define a completely positive map 
\[
\Gamma_q^\mathcal{U}(u): \mathcal{N}_\mathcal{U} \to \mathcal{N}_\mathcal{U}.
\]
such that $\Gamma_q^\mathcal{U}(u)\left( (u_N(h))^\bullet \right) = \left( u_N(u(h)) \right)^\bullet$. In particular, if $u$ is an isometry, $\Gamma_q^\mathcal{U}(u)$ is a $\ast$-homomorphism. The following lemma shall be crucial.
\begin{lemma}
	Let $E_\mathcal{U}: \mathcal{N}_\mathcal{U} \to \Gamma_q(H)$ be the unique conditional expectation.  Then	
	\begin{enumerate}
		\item For any contraction $u: H \to K$, $\Gamma_q^\mathcal{U}(u) \circ \pi_\mathcal{U} = \pi_\mathcal{U} \circ \Gamma_q(u)$. 
		\item
		\[
		E_\mathcal{U}\left(N^{-\frac{m}{2}} \sum_{\substack{j_1 \neq \ldots \neq j_m \\ 1 \leq j_k \leq N}} s_{j_1}(h_1) \ldots s_{j_m}(h_m)\right)^\bullet = W(h_1 \otimes_{} \cdots \otimes_{} h_m)
	\]\end{enumerate}
			(Here the indices are taken to be pair-wise not equal. )
	\label{ultrawick}
\end{lemma}

\begin{proof}
	1) is obvious. However, note that 1) implies $\pi_\mathcal{U}$. 
For 2), let 
	\[
	y_N = N^{-\frac{m}{2}} \sum_{\substack{j_1 \neq \cdots \neq j_m\\ 1 \leq j_k \leq N}} s_{j_1}(h_1) \cdots s_{j_m}(h_m)
	\]
	Let $y = E_\mathcal{U}(y_N)^\bullet$. Using i) and that the indices $j_k$ are all different, we know that 
	\[
	T_t(y) = E_\mathcal{U}(T_t^\mathcal{U}(y_N)^\bullet) = e^{-mt}E_\mathcal{U}(y_N) = e^{-mt}y
	\]
	Therefore, we have 
\[
	\tau(s(\bar{f}_{m'}) \cdots s(\bar{f}_1)y) = \tau(P_m(s(\bar{f}_{m'}) \cdots s(\bar{f}_1))y)
	\]
	where $P_m: \Gamma_q(H) \to \Gamma_q(H)$ denotes the projection defined by
	\[
	P_m(W(\xi)) = \left\{ \begin{array}{lr}
		W(\xi) & \mbox{if $\xi \in H^{\otimes_{}m}$}\\
		0 & \mbox{otherwise}
	\end{array}\right.
	\] 
	Therefore, we must only check the case where $m' = m$. 
	Using Theorem \ref{moment}, we obtain
	\begin{align*}
		&\tau(s(\bar{f}_{m})\cdots s(\bar{f}_1)y) = \tau_\mathcal{U}(s_\mathcal{U}(\bar{f}_{m})\cdots s_\mathcal{U}(\bar{f}_1)y_\mathcal{U}) \\
		&= \lim_N N^{-m } \sum_{k_1, \ldots, k_m} \sum_{j_1 \neq \cdots \neq j_m} \tau(s_{k_{m}}(\bar{f}_{m}) \cdots s_{k_1}(\bar{f}_1)s_{j_1}(h_1)\cdots s_{j_m}(h_m))\\
		&= \lim_N N^{- m} \sum_{k_1, \ldots, k_{m'}} \sum_{j_1 \neq \cdots \neq j_m} \sum_{\sigma \in P_2(2m)} q^{\iota(\sigma)} \prod_{\{\alpha,\beta \} \in \sigma} \langle \bar{f}_\alpha\otimes e_{k_\alpha}, h_\beta\otimes e_{j_\beta} \rangle
	\end{align*}
	Now simply counting the number of possible indices which make the inner product $\langle \bar{f}_\alpha \otimes e_{k_\alpha}, h_\beta \otimes e_{j_\beta} \rangle $ non-zero, we get
	\begin{align*}
		&\lim_N N^{-m} \sum_{k_1, \ldots, k_m} \sum_{j_1 \neq \cdots j_m} \sum_{\sigma \in P_2(2m)} q^{\iota(\sigma)} \prod_{\{\alpha, \beta\} \in \sigma} \langle \bar{f}_\alpha \otimes e_{k_\alpha}, h_\beta \otimes e_{j_\beta}   \rangle \\
		&= \lim_N N^{-m} \prod_{j = 0}^{m-1} (N-j)  \sum_{\sigma \in P_2(2m)} q^{\iota(\sigma)} \prod_{\{\alpha, \beta\} \in \sigma} \langle \bar{f}_\alpha, h_\beta \rangle \\
		&= \langle f_1 \otimes \cdots \otimes f_m, h_1 \otimes \cdots \otimes h_m \rangle 
	\end{align*}
It is easy to see that 
	\[
	f_1 \otimes \cdots \otimes f_m = P_m(s(f_1) \cdots s(f_m) \Omega)
	\]
	and so we get 
	\[
	\langle W(f_1 \otimes \cdots \otimes f_m), y \rangle = \langle f_1 \otimes \cdots \otimes f_m, h_1 \otimes \cdots \otimes h_m \rangle 
	\]
	Hence $y\Omega = h_1 \otimes \cdots \otimes h_m$ as required.
\end{proof}

\subsection{Bimodules}
Bimodules over von Neumann algebras were first defined and studied by Connes in his unpublished notes \cite{connesunpub} and were use by Connes and Jones in \cite{connesjones85} in order to define property (T) for von Neumann algebras. Specifically, for von Neumann algebras $\mathcal{M}$ and $\mathcal{N}$, an $\mathcal{M}$-$\mathcal{N}$-bimodule is a *-representation of $\mathcal{M} \otimes_{bin} \mathcal{N}^{op}$ on a Hilbert space $\mathcal{H}$. See \cite{effroslance} for the definition of the bin tensor norm. A simple and important example of a bimodule is the \emph{coarse} bimodule $L_{}^{2}(\mathcal{M}) \otimes L_{}^{2}(\mathcal{N})$ where $\mathcal{M}$ acts on the left on $L_{}^{2}(\mathcal{M})$ and $\mathcal{N}$ acts on the right on $L_{}^{2}(\mathcal{N})$. Just as for group representations, Connes and Jones gave the following definition of weak containment for these bimodules.

\begin{defi}
	Let $\mathcal{H}$ and $\mathcal{K}$ be $\mathcal{M}$-$\mathcal{N}$-bimodules. $\mathcal{H}$ is \emph{weakly contained} in $\mathcal{K}$, denoted by $\mathcal{H} \prec \mathcal{K}$, if for all $\varepsilon$, $\xi \in \mathcal{H}$, $F \subset \mathcal{M}$ finite, and $E \subset \mathcal{N}$ finite, there exists $\eta_1, \ldots, \eta_n \in \mathcal{K}$ such that
	\[
	|\langle \xi, x\xi y \rangle - \sum_{j=1}^n \langle \eta_j, x \eta_j y \rangle | < \varepsilon	
	\]
	for all $x \in F$ and $y \in E$.
\end{defi}

For two C*-algebras $\mathcal{A}$ and $\mathcal{B}$, denote the state space of the algebraic tensor product $\mathcal{A} \odot \mathcal{B}$ by $S(\mathcal{A} \odot \mathcal{A}) = S(\mathcal{A} \otimes_{max} \mathcal{B})$. In \cite{effroslance}, Effros and Lance show that if $\mathcal{M}$ and $\mathcal{N}$ are von Neumann algebras, $f \in S(\mathcal{M} \odot \mathcal{N})$ is such that $(x,y) \mapsto f(x \otimes y)$ is weak* continuous in each variable if and only if the maps
	\[
	T_f(x)(y) := f(x \otimes y)
	\]
	defines completely positive map $T_f: \mathcal{M} \to \mathcal{N}_\ast$ and $T_f^\ast: \mathcal{N} \to \mathcal{M}_\ast$ defines a completely positive map. We may define an element of $\varphi_\xi \in S(\mathcal{M} \odot \mathcal{N})$ from $\xi \in \mathcal{H}$ such that $\|\xi\|_{\mathcal{H}} = 1$ simply by $\varphi_\xi(x \otimes y) = \langle \xi, x\xi y \rangle$. These definitions give the following proposition.
	\begin{lemma}
		\label{weakcontHS}
		Let $\mathcal{M}$ and $\mathcal{N}$ be finite von Neumann algebras with separable predual and $\mathcal{H}$ be an $\mathcal{M}$-$\mathcal{N}$-bimodule. Then for the following, \emph{(2)} and \emph{(3)} are equivalent and \emph{(1)} implies \emph{(2)} and \emph{(3)}.
		\begin{enumerate}
			\item $T_\xi := T_{\varphi_\xi}$ extends to an element of $S_2(L_{}^{2}(\mathcal{M}), L_{}^{2}(\mathcal{N}))$.
			\item For $\xi \in \mathcal{H}$ such that $\|\xi\|_{\mathcal{H}} = 1$, $\varphi_\xi \in S(\mathcal{M} \odot \mathcal{N})$ is continuous with respect to the minimal tensor norm.
			\item $\mathcal{H} \prec L_{}^{2}(\mathcal{M}) \otimes L_{}^{2}(\mathcal{N})$.
		\end{enumerate}
			\end{lemma}
	\begin{proof}
		Note that $T_\xi(y)(x) = \varphi_\xi(x \otimes y) = \langle \xi, x\xi y \rangle$. Also note that the coarse bimodule $L_{}^{2}(\mathcal{M}) \otimes L_{}^{2}(\mathcal{N})$ is isomorphic to $S_2(L_{}^{2}(\mathcal{M}), L_{}^{2}(\mathcal{N}))$ by identifying simple tensors $\xi \otimes \eta \in L_{}^{2}(\mathcal{M}) \otimes L_{}^{2}(\mathcal{N})$ with rank one operators. The bimodule structure comes from pre-composing with operators from $\mathcal{M}$ and composing with operators from $\mathcal{N}$. Therefore we have
	\[
	L_{}^{2}(\mathcal{M}) \otimes L_{}^{2}(\mathcal{N}) \simeq L_{}^{2}(\mathcal{M} \bar{\otimes} \mathcal{N}^{op}) \simeq S_2(L_{}^{2}(\mathcal{M}), L_{}^{2}(\mathcal{N})).
	\]
	Let 
	\[
	\pi: \mathcal{M} \otimes_{bin} \mathcal{N} \to B(L_{}^{2}(\mathcal{M}) \otimes_{} L_{}^{2}(\mathcal{N})
	\]
	denote the representation described above which defines the bimodule structure.

(1) $\Rightarrow$ (2): 	Using these identifications, $T_\xi$ corresponds to $\zeta \in L_{}^{2}(\mathcal{M} \bar{\otimes} \mathcal{N}^{op})$ by 
	\[
	\tau_\mathcal{N}(T_\xi(x)y) = \tau_\mathcal{M} \otimes \tau_\mathcal{N}(\pi(x \otimes y) \zeta).
	\] 
	Since $\mathcal{M}$ and $\mathcal{N}$ are finite, we have that $\zeta \in L_{}^{1}(\mathcal{M} \bar{\otimes} \mathcal{N}^{op})$. Using the Kaplansky density theorem,  $L_{}^{1}(\mathcal{M} \bar{\otimes} \mathcal{N})$ embeds isometrically in $(\mathcal{M} \otimes_{min} \mathcal{N})^\ast$ since $\mathcal{M} \otimes_{min} \mathcal{N}$ is weak* dense in $\mathcal{M} \bar{\otimes} \mathcal{N}$. Therefore we have that $\|\varphi_\xi\|_{min^\ast} \leq \|T_\xi\|_{HS}$.

		(3) $\Rightarrow$ (2): If $\mathcal{H} \prec L_{}^{2}(\mathcal{M}) \otimes L_{}^{2}(\mathcal{N})$, using the definition, we have elements $\eta_1, \ldots, \eta_n \in L_{}^{2}(\mathcal{M}) \odot L_{}^{2}(\mathcal{N})$ such that 
		\[
		|\varphi_\xi(x \otimes y) - \sum_{j = 1}^n \varphi_{\eta_j}(x \otimes y)| < \varepsilon
		\]
		for all $x \in E$ finite, $y \in F$ finite, $\varepsilon> 0$. Since $\eta_j \in L_{}^{2}(\mathcal{M}) \odot L_{}^{2}(\mathcal{N})$, 
		\[
		\varphi_{\eta_j}(x \otimes y) = \langle \eta_j, \pi(x \otimes y) \eta_j  \rangle.
		\]
		Hence $\varphi_{\eta_j} \in B(L_{}^{2}(\mathcal{M}) \otimes L_{}^{2}(\mathcal{N})_\ast$. Therefore $\varphi_\xi$ is a cluster point of elements of the form $\sum_{j = 1}^n \varphi_{\eta_j}$. According to \cite{effroslance}, this implies that $\varphi_\xi$ is min-continuous.
					
		(2) $\Rightarrow$ (3): Suppose $\varphi_\xi$ is continuous with respect to the min-norm. We observe $\varphi_\xi \in (\mathcal{M} \otimes_{min} \mathcal{N})^\ast$ if it lifts to $B(L_{}^{2}(\mathcal{M}) \otimes L_{}^{2}(\mathcal{N}))^\ast$. Therefore, by \cite{effroslance}, we may write $\varphi_\xi$ as the limit of elements $\varphi_\eta \in B(L_{}^{2}(\mathcal{M}) \otimes L_{}^{2}(\mathcal{N}))_\ast$. This implies directly that $\mathcal{H} \prec L_{}^{2}(\mathcal{M}) \otimes L_{}^{2}(\mathcal{N})$. 
			\end{proof}

\section{Proof of the CCAP}
In this section, we shall prove that $\Gamma_q(H)$ has the w*CCAP. This is a crucial property to proving strong solidity given the result of Ozawa and Popa (\cite{OP1} Theorem 3.5) that every amenable subalgebra of a von Neumann algebra with the w*CCAP is weakly compact. This result is made significantly easier by using Theorem \ref{secquant} and Theorem 1 from \cite{nou}.
We begin by recalling Nou's result from Theorem 1 of \cite{nou}.

\begin{theorem}
	Let $K$ be a complex Hilbert space. Then for all $n \geq 0$ and for all $\xi \in B(K) \otimes_{\min} H^{\otimes n}$ we have
\begin{align*}
\max_{0 \leq k \leq n} \|(Id \otimes R^\ast_{n, k}(\xi)\| &\leq \|(Id \otimes W)(\xi)\|_{\min}\\
&\leq C_q(n+1)\max_{0 \leq k \leq n} \|(Id \otimes R^\ast_{n, k})(\xi)\|
\end{align*}
\end{theorem}

Let $X_n$ denote the $\ell^\infty$ direct sum of the spaces 
\[
H_c^{n-k} \otimes_{h} \bar{H}_r^{k}
\]
for $k$ ranging from 0 to $n$. This theorem means the map  $\Phi_n: H^{\otimes_{}n} \to \Gamma_q(H) $ defined by $\Phi_n(\xi) = W(\xi)$ has cb-norm less than $C_q(n+1)$.  Here the operator space structure on $H^{\otimes_{}n}$ is realized by demanding 
\[
R_{n,k}^\ast: H^{\otimes_{}n} \to X_n 
\]
have cb-norm $C_q = \prod_{n=1}^\infty (1 - q^n)^{-1}$. From here, we shall denote this operator space by $H_{Nou}^{\otimes_{}n}$. In other words, the following diagram commutes
\begin{equation*}
	\begin{tikzcd}
		H^{\otimes_{}n}_{Nou} \drar{\Phi_n} \dar[swap]{(R^\ast_{n,k})} \\ 
		 X_n \rar[swap]{(U_{n,k})} & \Gamma_q(H)
	\end{tikzcd}
\end{equation*}
Nou proves that $\|U_{n,k}\|_{cb} \leq C_q$ and so $\|\Phi_n\|_{cb} \leq C_q(n+1)$. Recall that $\overline{L_{}^{1}(\mathcal{M})}^\ast$ is completely isometric to $\mathcal{M}$ for a finite von Neumann algebra where the duality is with respect to the trace, i.e.
\[
\langle x, y \rangle = \tau(x^\ast y)
\]
which is consistent with the multiplication map from Lemma \ref{mult}. Let $P_n:\Gamma_q(H) \to \Gamma_q(H)$ be the projection defined by
\[
P_n(W(\xi)) = \left\{\begin{array}{lr}
	W(\xi) & \mbox{if $\xi \in H^{\otimes n}$}\\
	0 & \mbox{otherwise}
\end{array}\right.
\]
Following the now-standard argument in \cite{haagerup79}, our goal shall be to show that $\|P_n\|_{cb} < n^k$ for some natural number $k$. Once we have established this fact, we see that 
\[
\|x - \sum_{n=0}^N P_n T_t x\| \leq \|x - T_t(x)\|_{} + \sum_{n=0}^N e^{-nt} \|P_n(x)\|.
\] 
We may make the value on the right hand side as small as we would like by adjusting $N$ and $t$ thus there is a net $\varphi_\alpha = \sum_{n=0}^{N_\alpha} T_{t_\alpha}P_n$ which converges in the point-ultraweak topology to the identity. Also for any $\varepsilon >0$, we may choose $N_\alpha$ and $t_\alpha$ such that
\begin{align*}
	\|\varphi_\alpha\|_{cb} &\leq \|T_{t_\alpha}\|_{cb} + \sum_{n=N_\alpha+1}^{\infty} e^{-nt_\alpha}\|P_n\|_{cb} \\
	&\leq 1 + \sum_{n = N_\alpha + 1}^{\infty}e^{-nt_\alpha} n^k < 1+\varepsilon
\end{align*}
Let us define a map 
\[
\Psi_n: (H^{\otimes_{}n}_{Nou})^\ast \to \overline{L^1(\Gamma_q(H))}
\]
such that $\Psi_n(\xi) = W(\xi)$. Note that $(X_n)^\ast$ is the $\ell_1$ direct sum of the spaces
\[
\bar{H}_r^{\otimes_{}n-k} \otimes_{h} H_c^{\otimes k}
\]
as $k$ ranges from 0 to $n$. Similarly to \cite{nou}, we define the map 
\[
\overline{\Psi}_n: (X_n)^\ast \to \overline{L_{}^{1}(\Gamma_q(H))}
\]
such that
\[
\overline{\Psi}_n(h_1 \otimes_{} \cdots \otimes_{} h_{n-k} \otimes_{h} h_{n-k+1} \otimes_{} \cdots \otimes_{} h_n)_k = W(h_1 \otimes_{} \cdots \otimes_{} h_n).
\]
 Thus, since 
\[
R_{n,k}: \bar{H}_r^{\otimes_{}n-k} \otimes_{h} H_c^{\otimes_{}k} \to (H^{\otimes_{}n}_{Nou})^\ast 
\]
is defined by
\[
R_{n,k}(h_1 \otimes_{} \cdots \otimes_{} h_{n-k} \otimes_{h} h_{n-k+1} \otimes_{} \cdots \otimes_{} h_n) = h_1 \otimes_{} \cdots \otimes_{} h_n,
\]
the following diagram commutes.
\[
\begin{tikzcd}
	(X_n)^\ast \dar[swap]{(R_{n,k})} \drar{\overline{\Psi}_n} \\
	(H^{\otimes_{}n}_{Nou})^\ast \rar[swap]{\Psi_n} & \overline{L_{}^{1}(\Gamma_q(H))}
\end{tikzcd}
\]
From these definitions, we have the following lemma.

\begin{lemma}
	For $P_n$, $\Phi_n$, and $\Psi_n$ as defined above, 
	\[
	P_n = \Phi_n \circ \Psi_n^\ast
	\]
	\label{projfact}
\end{lemma}

\begin{proof}
	Let $\xi \in H^{\otimes_{}m}$ and $\eta \in \left( H^{\otimes_{}n}_{Nou} \right)$. We have
	\[
	\langle \eta,  \Psi_n^\ast(W(\xi))\rangle = \langle \Psi_n(\eta), W(\xi) \rangle = \langle W(\eta), W(\xi) \rangle = \langle \eta, \xi \rangle 
	\]
	Therefore, 
	\[
	\Psi_n^\ast(W(\xi)) = \left\{ \begin{array}{lr}
		\xi & \mbox{if $m = n$} \\
		0 & \mbox{if $m \neq n$}
	\end{array}\right.
	\]
	Since $\Phi_n(\xi) = W(\xi)$, this completes the proof.
\end{proof}
Therefore, $\|P_n\|_{cb} \leq \|\Phi_n\|_{cb} \|\Psi_n\|_{cb}$, and so we must show that $\|\Psi_n\|_{cb} < p(n)$ for some polynomial $p$. Since $R_{n,k}$ is a quotient map by definition, we must show that $\|\bar{\Psi}_n\|_{cb} < p(n)$. With this in mind, we have the following proposition.

\begin{prop}
	Let $\beta_{n,k}: \bar{H}_r^{\otimes_{}n-k} \otimes_{h} H_c^{\otimes_{}k} \to \overline{L_{}^{1}(\Gamma_q(H))}$ be defined by
	\[
	\beta_{n,k}(h_1 \otimes_{} \cdots \otimes_{} h_{n-k} \otimes_{h} h_{n-k+1} \otimes_{} \cdots \otimes_{} h_n) = W(h_1 \otimes_{} \cdots \otimes_{} h_n).
	\]
	so that $\bar{\Psi}_n = (\beta_{n,k})$. We have $\|\beta_{n,k}\|_{cb} \leq C_q$. 	
	\label{coordbound}
\end{prop}

The proof of this proposition will require several lemmas. However, note that it follows directly from this proposition that $\|P_n\|_{cb} < C_q n^2$. 

\begin{lemma}
	Let $v_{n,k}: \bar{H}_r^{\otimes n-k} \otimes_h H_c^{\otimes k} \to \overline{L^1(\Gamma_q(H))}$ be the map
	\begin{align*}	
		&v_{n,k}((h_1 \otimes \ldots \otimes h_{n-k}) \otimes_h (h_{n -k + 1} \otimes \ldots \otimes h_n)) \\
		&W(h_1 \otimes \ldots \otimes h_{n-k})W(h_{n-k +1} \otimes \ldots \otimes h_n)
	\end{align*}
		Then $\|v_{n, k}\|_{cb} = 1$.
	\label{little}
\end{lemma}

\begin{proof}
	$v_{n, k} = m \circ (W \otimes W)$ where $W$ is the Wick word and $\bar{m}: \overline{L^2_r(\Gamma_q(H))} \otimes L_{c}^{2}(\Gamma_q(H)) \to \overline{L^1(\Gamma_q(H))}$ is the canonical multiplication map from Lemma \ref{mult}, which is completely contractive. Note that $W: F_q(H) \to L^2(\Gamma_q(H))$ a unitary transformation.
\end{proof}

\begin{lemma}
	Define $w_{n, k}^j: \bar{H}_r^{\otimes n-k} \otimes_h H_c^{\otimes k} \to \overline{L^1(\Gamma_q(H))}$ by 
	\[
	w_{n, k}^j = v_{n-2j, k-j} \circ (Id_{n-k-j} \otimes m_j \otimes Id_{k-j}) \circ (R_{n-k, j}^\ast \otimes R_{k, k-j}^\ast)
	\]
	where $Id_k$ is the identity on $H^{\otimes k}$ and $m_j: \bar{H}_r^{\otimes j} \otimes H_c^{\otimes j} \to \mathbb{C}$ is simply a duality bracket pairing.

	Then $\|w_{n, k}^j\|_{cb}$ is bounded by a constant depending only on $q$.
	\label{little2}
\end{lemma}
\begin{notation}
	For an n-tensor $\xi = h_1 \otimes \cdots \otimes h_n$ and a subset $A = \{\iota_1, \ldots, \iota_k\} \subset \{1, \ldots, n\}$, denote by $\xi_A$ the tensor $h_{\iota_1} \otimes \cdots \otimes h_{\iota_k}$.
\end{notation}

	\begin{rem}
	As the maps $w_{n, k}^j$ are crucial to our argument, we further describe the image of an element of $H_r^{\otimes n-k} \otimes_h H_c^{\otimes k}$. 
		\begin{align*}	
			&w_{n,k}^j (\xi_{n-k} \otimes_h \xi^k) = 
			\sum_{\substack{A \subset \{1, \ldots n-k\} \\ |A| = j}} \sum_{\substack{B \subset\{n-k+1, \ldots n\} \\ |B| = j}} q^{\iota(A) + \iota(B)} \langle \xi_{n-k, A}, \xi_B^{k} \rangle_q W(\xi_{n-k, A^c}) W\left( \xi_{B^c}^k \right)\\
		&= \sum_{\substack{A \subset \{1, \ldots, n-k\} \\ |A| = j}} \sum_{\substack{B \subset \{n-k+1, \ldots, n\}\\ |B| = j}} \sum_{\sigma \in S_j} q^{\iota(A) + \iota(B) + \iota(\sigma)} \prod_{s = 1}^j \langle h_{a_s},h_{b_{\sigma(s)}}  \rangle W(\xi_{n-k, A^c}) W(\xi_{B^c}^k)
			\end{align*}
			where $a_s$ is the $s$th element of $A$ and likewise for $b_s$. $\iota(A)$ and $\iota(B)$ are the number of inversions of the corresponding right cosets from Observation \ref{rnk}.

\end{rem}

\begin{proof}[Proof of Lemma \ref{little2}]
	From Lemma \ref{little} we know that $\|v_{n - 2j, k-j}\|_{cb} \leq 1$ and it is clear that the middle term in the composition defining $w_{n, k}^j$ is completely contractive. It is shown in \cite{nou} that $\|R_{n, k}\|_{cb} \leq C_q$ where $C_q = \prod_{j\geq 1} (1 - q^{j})^{-1}$. Therefore, it is clear that $\|w_{n, k}^j\|_{cb} \leq C_q^2$.
\end{proof}

\begin{lemma}
	Let $\xi = h_1 \otimes \cdots \otimes h_n$, $\eta = k_1 \otimes \cdots \otimes k_m$, and $\theta = f_1 \otimes \cdots \otimes f_\ell$. We have
	\[
	\tau(W(\xi)W(\eta)W(\theta)) = \sum_{A, B, C} q^{\iota(A) + \iota(B) + \iota(C)} \langle \xi_A, \eta_B \rangle \langle \xi_{A^c}, \theta_C \rangle \langle \eta_{B^c}, \theta_{C^c} \rangle 
	\]
	where $A \subset \{1, \ldots, n\}$, $B \subset \{1, \ldots, m\}$, and $C \subset \{1, \ldots, \ell\}$, such that $|A| = |B|$, $|A^c| = |C|$, and $|B^c| = |C^c|$ and the sum ranges over all such subsets.
	\label{threeprod}
\end{lemma}

\begin{proof}
	We note first that $\ell + m + n$ must be even. Let $\alpha$, $\beta$, and $\gamma$ be multi-indices of lengths $n$, $m$, and $\ell$ respectively such that they are each pairwise not equal (i.e. $\alpha_j \neq \alpha_k$ for $j\neq k$). Let $x_\alpha = s_{\alpha_1}(h_1)\cdots s_{\alpha_n}$, $y_\beta = s_{\beta_1}(k_1) \cdots s_{\beta_m}(k_m)$, and $z_\gamma = s_{\gamma_1}(f_1) \cdots s_{\gamma_\ell}(f_\ell)$. We apply Lemma \ref{ultrawick} and Theorem \ref{moment}.
	\begin{align*}
		&\tau(W(\xi)W(\eta)W(\theta)) = \lim_N N^{(n + m +\ell)/2} \sum_{\alpha, \beta, \gamma}\tau(x_\alpha y_\beta z_\gamma)\\
		&= \lim_N N^{(n + m + \ell)/2} \sum_{\alpha, \beta, \gamma} \sum_{\sigma \in P_2(n + m + \ell)} q^{\iota(\sigma)} \prod_{\{\alpha_r, \beta_s\}\in \sigma} \langle h_r, k_2 \rangle \prod_{\{\alpha_r, \gamma_s\} \in \sigma} \langle h_r, f_s \rangle \prod_{\{\beta_r, \gamma_s\} \in \sigma} \langle k_r, f_s \rangle.
	\end{align*}
	Identifying $A$, $B$, and $C$ as the subsets of multi-indices $\alpha$, $\beta$, and $\gamma$ so that $\{\alpha_j, \beta_k\} \in \sigma$ if and only if $\alpha_j \in A$ and $\beta_k \in B$ and $\{\alpha_j, \gamma_k\} \in \sigma$ if and only if $\alpha_j \in A^c$ and $\gamma_k \in C$, from Observation \ref{rnk}, we get the result.
	
\end{proof}
\begin{notation}
	We shall denote by $P_{1, 2}(m)$ the set of all partitions of $\{1, \ldots, m\}$ whose parts are no larger than two and set 
\[
P_{1, 2}^k(m)= \{\sigma \in P_{1, 2} |  \{i,j\} \in \sigma \Rightarrow i \in \{1, \ldots, m-k\} \mbox{ and } j \in \{m-k+1, \ldots m\}\}.
\]
For $\sigma \in P_{1, 2}(m)$, $\iota(\sigma)$ will denote the number of ``crossings'' of $\sigma$, i.e.
\[
\iota(\sigma) = |\{\{i, j\}, \{k, \ell\} \in \sigma: i < k < j < \ell\}| + |\{\{i, j\}, \{k\} \in \sigma:  i < k < j\}|.
\] 
We shall denote by $P_{1,2}^{j,k}(n)$ the subset of $P_{1,2}^k(n)$ with exactly $j$ pairs.
\end{notation}

\begin{exam}
	Let 
	\[
	\sigma = \{ \{1\},\{2, 5\}, \{3\}, \{4,7\}, \{6\}, \{8\}\} \in P_{1,2}^{2, 4}(8).
	\]
	We may represent $\sigma$ using the following figure. 
	
	\begin{figure}[h!]
		\begin{center}
			\includegraphics[scale = 0.75]{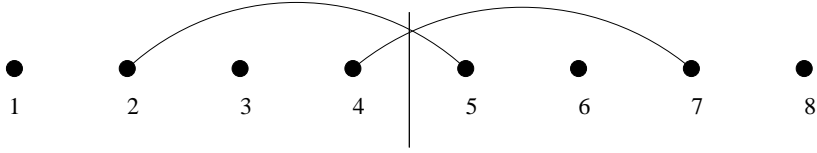}
		\end{center}
		\caption{$\sigma$, $\iota(\sigma) = 3$}
		\label{fig:base}
	\end{figure}
	We can see that $\iota(\sigma) = 3$ since $\{2, 5\}$ crosses the singleton $\{3\}$, $\{4, 7\}$ crosses the singleton $\{6\}$, and the pairs $\{2, 5\}$ and $\{4, 7\}$ cross.
\end{exam}
Before we state the key proposition, we shall need to study two ``color'' operators. In the case of $y_j$, $j$ tensors are given an arbitrary new color, whereas for $z_j$, $j$ tensors are given a new color in decreasing order of colors. 

\begin{defi}
	Let $H$ be a real Hilbert space and $\{e_\ell\}_{\ell = 0}^j$ an orthonormal basis of $\ell_{j+1}^{2}(\mathbb{R})$, we define 
	\[
	y_j: H^{\otimes n} \to (H \otimes \ell_{j+1}^{2}(\mathbb{R}))^{\otimes n}
	\]
	by
	\[
	y_j(h_1 \otimes \cdots \otimes h_n) = \sum_{\substack{A \subset \{1, \cdots, n\} \\ |A| = j}} \sum_{\substack{f:A \to \{1, \cdots, j\} \\ f|_A \mbox{ {\tiny a bijection}} \\ f(A^c) = \{0\}}} (h_1 \otimes e_{f(1)}) \otimes \cdots \otimes (h_n \otimes e_{f(n)}).
	\]
	Also, we define
	\[
	z_j: H^{\otimes n} \to (H \otimes \ell_{j+1}^{2}(\mathbb{R}))^{\otimes n}
	\]
	by
	\[
	z_j(h_1 \otimes \cdots \otimes h_n) = \sum_{\substack{A \subset \{1, \ldots, n\} \\ |A| = j}} (h_1 \otimes e_{f_A(1)}) \otimes \ldots \otimes (h_n \otimes e_{f_A(n)})
	\]
	where 
	\[
	f_A(\ell) = \left\{ \begin{array}{ll}
		0 & \mbox{ if $\ell \notin A$}\\
		j-k+1 & \mbox{ if $\ell$ is the kth largest element of $A$}
	\end{array}.\right.
	\]
\end{defi}

\begin{notation}
	For $\{e_\ell\}_{\ell = 0}^j$ and orthonormal basis of $\ell_{j+1}^{2}(\mathbb{R})$, let 
	\[
	E_j: \Gamma_q(H \otimes \ell_{j+1}^{2}(\mathbb{R})) \to \Gamma_q(H \otimes \ell_{j}^{2}(\mathbb{R}))
	\]
	be the conditional expectation given by $E_j = \Gamma_q(P_j)$, where $P_j$ is the projection such that $P_j(e_j) = 0$ and $P_j(e_\ell) = e_\ell$ for $\ell \neq j$. 
\end{notation}

	\begin{defi}
		Let $\sigma_\emptyset \in P_{1,2}^k(n)$ be the singleton partition, and let $\sigma_j \in P_{1,2}^{j,k}(n)$ and $\sigma_{j-1}\in P_{1,2}^{j-1,k}(n)$ be such that $\sigma_j \setminus \sigma_{j-1} = \{\ell_1, \ell_2\}$ where if $\{k_1, k_2\} \in \sigma_{j-1}$, then $\ell_1 < k_1$. We define a new function $\iota'$ on $P_{1,2}^k(n)$ recursively by
		\begin{enumerate}
			\item $\iota'(\sigma_\emptyset) = 0$.
			\item $\iota'(\sigma_j) = \iota'(\sigma_{j-1}) +$ 
				\[
				|\left\{ \{m\} \in \sigma_{j-1}: \ell_1 < m < \ell_2 \right\}| + 2|\left\{ \{k_1, k_2\} \in \sigma_{j-1}: \ell_1 < k_1 < k_2 < \ell_2 \right\}|
				\]
		\end{enumerate}
	\end{defi}

	\begin{exam}
		\label{examnewiota}
		Let 
		\[
		\sigma = \{ \{1, 6\}, \{2, 5\}, \{3\}, \{4\}, \{7\}, \{8\}\}
		\]
		shown in Figure \ref{fig:partition}.   
		\begin{figure}[h!]
			\begin{center}
				\includegraphics[scale = 0.75]{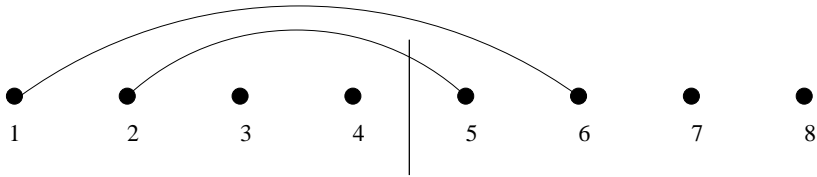}	
			\end{center}
			\caption{$\iota'(\sigma) = 6$, $\iota(\sigma) = 4$}
			\label{fig:partition}
		\end{figure}

		We can see that $\iota'(\sigma) = 6$ since $\{1,6\}$ ``contains'' $\{2, 5\}$, $\{3\}$, and $\{4\}$, and $\{2, 5\}$ ``contains'' $\{3\}$ and $\{4\}$. 
	\end{exam}
	We now have the following lemma
\begin{lemma}
	\label{wnkj}
		For $w_{n,k}^j$, $y_j$, and $z_j$ as above, we have
		\[
		q^{ {j \choose 2}} w_{n,k}^j(\xi_{n-k} \otimes \xi^k) = E_\mathcal{U} \left(E_1 \cdots E_j\left(z_j(u_N(\xi_{n-k})) y_j(u_N(\xi^k))\right)\right)
		\]
	\end{lemma}

	\begin{proof}
	For the left hand side, we have that
	\begin{align*}
		&q^{ {j \choose 2}} w_{n, k}^j(\xi_{n-k} \otimes \xi^k) \\
		&= q^{ {j \choose 2}} \sum_{ \substack{ A \subset \{1, \ldots, n-k\} \\ |A| = j}} \sum_{ \substack{B \subset \{1, \ldots, k\} \\ |B|= j}} q^{\iota(A) + \iota(B)} \langle \xi_{n-k, A}, \xi_B^k \rangle_q W(\xi_{n-k, A^c}) W(\xi_{B^c}^k)\\
		&= \sum_{ \substack{ A \subset \{1, \ldots, n-k\} \\ |A| = j}} \sum_{ \substack{B \subset \{1, \ldots, k\} \\ |B|= j}} \sum_{\sigma \in S_j} q^{ \iota(A) + \iota(B) + \iota(\sigma) + {j \choose 2}} \prod_{s = 1}^j \langle h_{a_s}, h_{b_\sigma(s)} \rangle W(\xi_{n-k, A^c}) W(\xi_{B^c}^k)\\
		&= \sum_{\rho \in P_{1,2}^{j, k}(n)} q^{\iota(A) + \iota(B) + \iota(\sigma) + {j \choose 2}} \prod_{\{\ell_1, \ell_2\} \in \rho} \langle h_{\ell_1}, h_{\ell_2} \rangle W(\xi_{n-k, \rho}) W(\xi_\rho^k),
	\end{align*}
	where $\xi_\rho$ denotes $\xi$ with the pairs of $\rho$ removed.
	For the right hand side, we have that
	\begin{align*}
		&E_\mathcal{U} \left( E_1\cdots E_j\left( z_j(u_N(\xi_{n-k}))y_j(u_N(\xi^k) \right) \right)\\ &= E_\mathcal{U} \left(  \sum_{ \substack{ A \subset \{1, \ldots, n-k\} \\ |A| = j}} \sum_{ \substack{B \subset \{1, \ldots, k\} \\ |B|= j}} \sum_{ \substack{g: B \to \{1, \ldots, j\} \\ g \mbox{ a bijection}}} E_1\cdots E_j\left(u_N(\xi_{n-k}^A) u_N(\xi^{k, (B, g)})\right) \right),
	\end{align*}
	where $\xi_{n-k}^A$ is corresponding tensor in image of $z_j$ in $(H \otimes \ell_{j+1}^{2}(\mathbb{R}))^{\otimes n-k}$ and similarly for $\xi^{k, (B,g)}$. We examine 
	\[
	E_1\cdots E_j\left( u_N(\xi_{n-k}^A)u_N(\xi^{k, (B, g)}) \right)
	\]
	for fixed $A$, $B$, and $g$. For this term and fixed $N$, we have
	\begin{align*}
		&E_1\cdots E_j\left( u_N(\xi_{n-k}^A)u_N(\xi^{k, (B, g)}) \right)\\
		&=  q^{\ell_2 - \ell_1-1} \langle h_{\ell_1} \otimes e_{\alpha_{\ell_1}}, h_{\ell_2} \otimes e_{\beta_{\ell_2}}\rangle E_1 \cdots E_{j-1} \left( N^{-1}u_N(\xi_{n-k}^{A \setminus \{\ell_1\}}) u_N(\xi^{k, (B \setminus \{\ell_2\}, g')})  \right)\\
		&=  q^{\iota'(\rho)} \prod_{\{\ell_1, \ell_2\} \in \rho} \langle h_{\ell_1} \otimes e_{\alpha_{\ell_1}}, h_{\ell_2} \otimes e_{\beta_{\ell_2}} \rangle N^{-j}u_N(\xi_{n-k, \rho}) u_N(\xi_\rho^k),
	\end{align*}
	where $\rho \in P_{1, 2}^{j,k}(n)$ denotes the partition whose pairs are given by $\{\ell_1, \ell_2\} \in \rho \Rightarrow f_A(\ell_1) = g(\ell_2) \neq 0$. The factor $N^{-j}$ comes from the fact that we are shortening the tensors so we must compensate for the factor of $N$ in the formula for $u_N(\xi)$. To see that $\iota'(\rho)$ is the appropriate power of $q$, after applying $E_j$ we have a term $q^{\ell_2 - \ell_1 - 1}$. However, we have removed $h_{\ell_1}$ and $h_{\ell_2}$ from the tensor, so we must compensate for the remaining pairs in $\rho$ which cross $\{\ell_2\}$. Letting $\rho = \rho_j$ and the partition with the same pairs of $\rho$ except $\{\ell_1, \ell_2\}$ be $\ell_2$, we get,
	\begin{align*}
		&\iota'(\rho_j) - \iota'(\rho_{j-1}) = \ell_2 - \ell_1 - 1 - |\{ \{k_1, k_2\} \in \rho_2: k_1 < \ell_2 < k_2\}|\\
		&= |\{\{m\}\in \rho_j: \ell_1 < m < \ell_2\}| + 2|\{\{k_1, k_2\} \in \rho_j: \ell_1 < k_1 < k_2 < \ell_2\}|
	\end{align*}
	where the last equality follows by simply observing that those are the elements remaining after removing the elements of $\{ \{k_1, k_2\} \in \rho_2: k_1 < \ell_2 < k_2\} $. Let's reconsider Example \ref{examnewiota}.
	\begin{exam}
		Recall we used
	\[
		\sigma = \{ \{1, 6\}, \{2, 5\}, \{3\}, \{4\}, \{7\}, \{8\}\}
		\]

		After applying $y_2$ and $z_2$, we have colored the indices which appear in the two pairs of $\sigma$ as shown in Figure \ref{fig:newiota1}. 

		\begin{figure}[h!]
			\begin{center}
				\includegraphics[scale = 0.75]{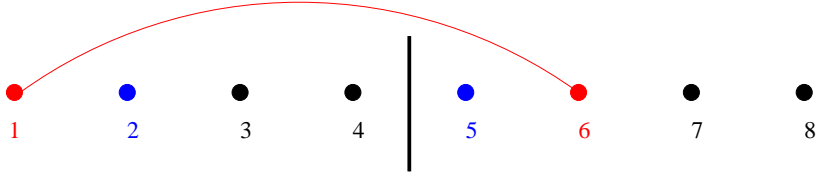}	
			\end{center}
			\caption{$z_j \otimes y_j$}
			\label{fig:newiota1}
		\end{figure}

		We then apply the conditional expectation $E_2$, and the result is shown in Figure \ref{fig:newiota2}.

		\begin{figure}[h!]
			\begin{center}
				\includegraphics[scale=0.75]{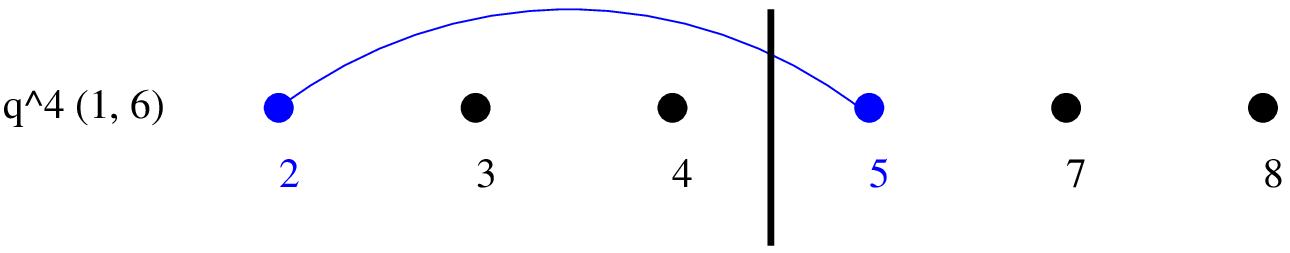}
			\end{center}
			\caption{$E_2$}
			\label{fig:newiota2}
		\end{figure}

		Finally, we apply $E_1$, and the result is shown in Figure \ref{fig:newiota3}.

		\begin{figure}[h!]
			\begin{center}
				\includegraphics[scale=0.75]{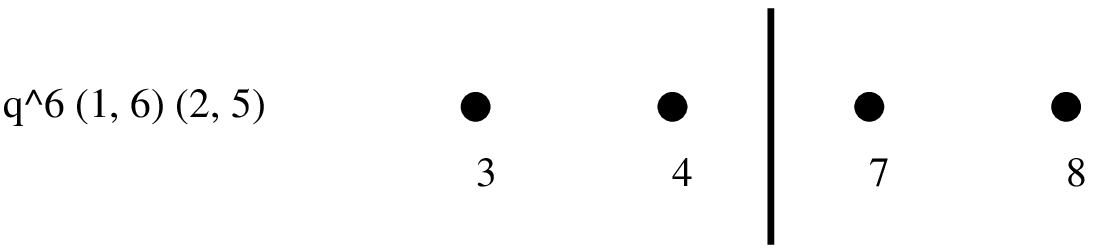}
			\end{center}
			\caption{$E_1$}
			\label{fig:newiota3}
		\end{figure}

		This coincides with $\iota'(\sigma)$ as shown in Example \ref{examnewiota}. 

	\end{exam}
	We now apply $E_\mathcal{U}$ to an element such as
	\[
	 q^{\iota'(\rho)} \prod_{\{\ell_1, \ell_2\} \in \rho} \langle h_{\ell_1} \otimes e_{\alpha_{\ell_1}}, h_{\ell_2} \otimes e_{\beta_{\ell_2}} \rangle N^{-j}u_N(\xi_{n-k, \rho}) u_N(\xi_\rho^k).
	\]
	From Theorem \ref{ultrawick} we get
	\begin{align*}
		&E_\mathcal{U} q^{\iota'(\rho)} \prod_{\{\ell_1, \ell_2\} \in \rho} \langle h_{\ell_1} \otimes e_{\alpha_{\ell_1}}, h_{\ell_2} \otimes e_{\beta_{\ell_2}} \rangle N^{-j} u_N(\xi_{n-k, \rho}) u_N(\xi_\rho^k)\\ 
		&= q^{\iota'(\rho)} \prod_{\{\ell_1, \ell_2\} \in \rho} \langle h_{\ell_1}, h_{\ell_2} \rangle E_\mathcal{U} u_N(\xi_{n-k, \rho})u_N(\xi_\rho^k)\\
		&= q^{\iota'(\rho)} \prod_{\{\ell_1, \ell_2\}\in \rho} \langle h_{\ell_1},  h_{\ell_2} \rangle W(\xi_{n-k, \rho}) W(\xi_\rho^k)
	\end{align*}
	Here the factor $N^{-j}$ is offset since for each pair $\{\ell_1, \ell_2\} \in \rho$, $\alpha_{\ell_1} = \beta_{\ell_2}$ in order for the inner product to be non-zero. For each of the $j$ partitions, there are $N$ possibilities for indices which match. This gives a factor of $N^j$. From Theorem \ref{ultrawick}, we get that
	\[
	E_\mathcal{U} u_N(\xi)u_N(\eta) = W(\xi)W(\eta)
	\]
	Now since summing over all possible $A$, $B$, and $g$ such that $|A| = |B| = j$ is the same as summing over all elements of $P_{1, 2}^{j, k}(n)$, we see that the right hand side is equal to 
	\[
	\sum_{\rho \in P_{1, 2}^{j, k}(n)} q^{\iota'(\rho)} \prod_{\{\ell_1, \ell_2\} \in \rho} \langle h_{\ell_1}, h_{\ell_2} \rangle W(\xi_{n-k, \rho})W(\xi_\rho^k)
	\]
	Therefore, to prove the lemma, we must only check that $\iota'(\rho) = \iota(A) + \iota(B) + \iota(\sigma) + { j \choose 2}$. Recall that $A$ and $B$ are subsets of $\{1, \ldots, n-k\}$ and $\{n-k+1, \ldots, n\}$ respectively, and that $|A| = |B| = j$. Recall also that $\iota(A)$ and $\iota(B)$ are given by associating $A$ and $B$ to cosets in $S_{n-k}/S_{n-k+j} \times S_j$ and $S_k/S_j \times S_{k-j}$ respectively as in Observation \ref{rnk}. The permutation $\sigma \in S_j$ identifies how to pair elements of $A$ with elements of $B$ so that we may associate these three data with an element of $P_{1, 2}^{j, k}(n)$. Therefore, for $j=0$, we have

	\[
	\iota'(\rho_0) = \iota(\emptyset) + \iota(\emptyset) + \iota(\sigma_\emptyset) + {0 \choose 2} = 0
	\]
	where $\rho_0$ is the element of $P_{1, 2}^{k}$ containing no pairs. 
	
	Now let $\rho_j$ and $\rho_{j-1}$ be such that $\rho_j \setminus \rho_{j-1} = \{\ell_1, \ell_2\}$ where if $\{k_1, k_2\} \in \rho_{j-1}$, $\ell_1 < k_1$. Let $c_1 =\iota(\rho_j) - \iota(\rho_{j-1})$, and $c_2 = \iota'(\rho) - \iota'(\rho_{j-1})$. By our inductive hypothesis, we assume
	
	\[
	\iota'(\rho_{j-1}) = \iota(A_{j-1}) + \iota(B_{j-1}) + \iota(\sigma_{j-1}) + {j-1 \choose 2},
	\]
	where $A_{j-1}$, $B_{j-1}$, and $\sigma_{j-1}$ are associated to $\rho_{j-1}$ as described above. Let $A_j$, $B_j$, and $\sigma_j$ be associated to $\rho_j$ similarly. Then we have
	\begin{align*}	
		&\iota(A_j) + \iota(B_j) - \iota(A_{j-1}) - \iota(B_{j-1}) \\
		&= |\{\{m\} \in \rho_j : \ell_1 < m < \ell_2\}| + |\{ \{k_1, k_2\} \in \rho_j: \ell_1 < k_1 < \ell_2 < k_2\}|\\
		&=c_1 - (j - 1 - \frac{c_2 - c_1}{2}),
		\end{align*}
	since $\ell_1$ and $\ell_2$ must cross all of the singletons $\{m\}$ such that $\ell_1 < m < \ell_2$. However we subtract the term $j-1 - \frac{c_2 - c_1}{2}$ since these are the pairs $\{k_1, k_2\}$ such that $k_1 < \ell_2 < k_2$ which had to cross the singleton $\ell_2$ in $\iota(B_{j-1})$. Since $\frac{c_2 - c_1}{2}$ is the number of pairs $\{k_1, k_2\} \in \rho_j$ such that $\ell_1 < k_1 < k_2 < \ell_2$, and there are $j-1$ pairs in $\rho_{j-1}$, we get that there are $j-1 - \frac{c_2 - c_1}{2}$ such pairs. For the permutations, we get
	\[
	\iota(\sigma_j) - \iota(\sigma_{j-1}) = \frac{c_2 - c_1}{2}
	\]
	since we must multiply the element $(1, \frac{c_2 - c_1}{2}) \in S_j$ onto $\sigma_{j-1}$ (where $\sigma_{j-1}$ is viewed as an element of the subgroup $S_1 \times S_{j-1}$). Since $\iota$ is multiplicative on $S_j$, we get the equality above. Therefore
		\begin{align*}
		&\iota(A_j) + \iota(B_j) + \iota(\sigma_j) + {j \choose 2} - \iota'(\rho_{j-1}) \\
		&= c_1 - \left(j - 1 - \left(\frac{c_2 - c_1}{2}\right)\right) + \frac{c_2 - c_1}{2} + {j \choose 2} - {j-1 \choose 2}\\
		&= c_1 + c_2 - c_1 - (j-1) + j-1 = c_2
	\end{align*}
	which finishes the proof.
	\end{proof}
Now it is time for an example which clarifies this inductive step.
	\begin{exam}
		Let 
		\[
		\rho_2 = \{\{1\}, \{2, 5\}, \{3\}, \{4, 7\}, \{6\}, \{8\}\},
		\]
		and
		\[
		\rho_3 = \{\{1, 6\}, \{2, 5\}, \{3\}, \{4, 7\}, \{8\}\}.
		\]
		Then for $\rho_2$, $A_2 = \{2, 4\}$, $B_2 = \{5, 7\}$, and $\sigma_2 = 1$, as shown in Figure \ref{fig:rearrangeold}.  

		\begin{figure}[h!]
			\begin{center}
				\includegraphics[scale = 0.75]{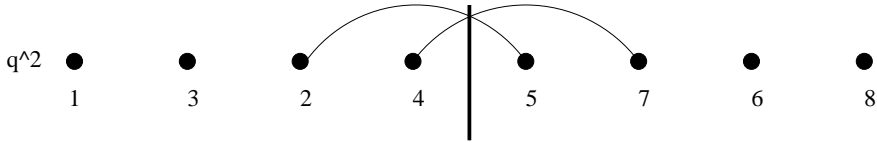}
			\end{center}
			\caption{$\iota(A_2) = 1$, $\iota(B_2) = 1$, $\iota(\sigma_2) = 0$}
			\label{fig:rearrangeold}
		\end{figure}

		For $\rho_3$, $A_3 = \{1, 2, 4\}$, $B_3 = \{5, 6, 7\}$, so after arranging $A$ and $B$ as seen in Figure \ref{fig:rearrangenew}.

		\begin{figure}[h!]
			\begin{center}
				\includegraphics[scale = 0.75]{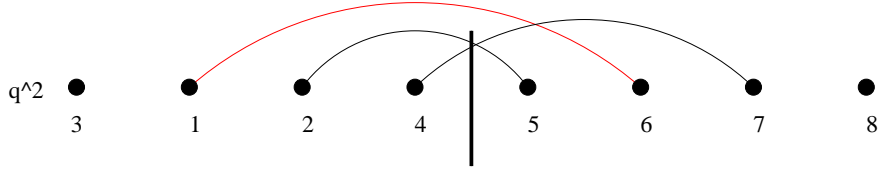}
			\end{center}
			\caption{$\iota(A_3) = 2$, $\iota(B_3) = 0$}
			\label{fig:rearrangenew}
		\end{figure}
		
		Now we must apply the transposition $\sigma_3 = (5, 6)$, and the result can be seen in Figure \ref{fig:sigmanew}. 
		
		\begin{figure}[h!]
			\begin{center}
				\includegraphics[scale=0.75]{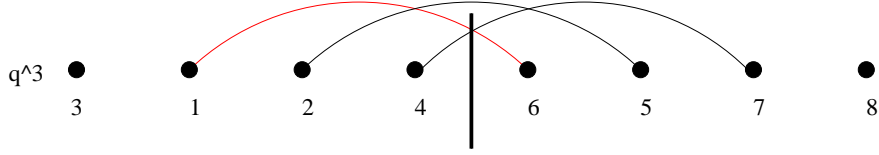}
			\end{center}
			\caption{$\iota(\sigma_3) = 1$}
			\label{fig:sigmanew}
		\end{figure}

		Adding ${2 \choose 2} = 1$ to $\iota(A_2) +\iota(B_2) + \iota(1)$, we get $\iota'(\rho_2) = 3$, and adding ${3 \choose 2} = 3$ to $\iota(A_3) + \iota(B_3) + \iota(\sigma_3)$, we get $\iota'(\sigma_3) = 6$. 		
	\end{exam}

\begin{prop}
	We have for $\beta_{n,k}$ and $w_{n,k}^j$ as above
	\[
	\beta_{n, k} = \sum_{j = 0}^{k \vee n-k} (-1)^j q^{ {j \choose 2} }w_{n, k}^j
	\]
	\label{big}
\end{prop}
\begin{note}
	Here we are defining $w_{n, k}^0 := v_{n, k}$.
\end{note}
	\begin{proof}
	Let $\xi = h_1 \otimes \cdots \otimes h_n$, and let $\xi_{n-k}$  and $\xi^k $ be as before. From Lemma \ref{threeprod}, we get that 
	\[
	W(\xi_k)W(\xi^{n-k}) = \sum_{\sigma \in P_{1,2}^k(n)} q^{\iota(\sigma)} \prod_{\{i, j\}\in \sigma} \langle h_i, h_j \rangle W(\xi)_\sigma
	\]
	We can see this by examining $\langle W(\eta), W(\xi_k)W(\xi^{n-k} \rangle = \tau(W(\eta)^\ast W(\xi_k)W(\xi^{n-k}))$ with the formula from Lemma \ref{threeprod}. 
	Recall from Lemma \ref{ultrawick} that 
	\[
	W(h_1 \otimes \ldots \otimes h_n) = E_\mathcal{U} ((N^{-n/2} \sum_{j_1\neq \cdots \neq j_n } s_{j_1}(h_1) \cdots s_{j_n}(h_n))^\bullet)
	\]

	For $\xi = h_1 \otimes \cdots \otimes h_n$, let 
	\[
	u_N(\xi) = N^{-\frac{n}{2}} \sum_{\alpha} s_{\alpha_1}(h_1) \cdots s_{\alpha_n}(h_n)
	\] 
	From Lemma \ref{wnkj}, we have that 
	\begin{align*}
		\sum_{j=0}^{k\vee n-k} (-1)^j q^{ {j \choose 2}} w_{n,k}^j(\xi_{n-k} \otimes \xi^k) &= \sum_{j=0}^{k \vee n-k} (-1)^j E_\mathcal{U} E_1 \cdots E_j\left(z_j(u_N(\xi_{n-k}))y_j(u_N(\xi^k))\right)\\
		&= \sum_{j=0}^{k \vee n-k} (-1)^j \sum_{\rho \in P_{1, 2}^{j, k}(n)} q^{\iota'(\rho)} \prod_{\{\ell_1, \ell_2\} \in \rho} \langle h_{\ell_1}, h_{\ell_2} \rangle E_\mathcal{U}\left(u_N(\xi_{n-k, \rho})u_N(\xi_\rho^k)\right)
	\end{align*}
	Now from Theorem \ref{ultrawick} and Lemma \ref{threeprod}, we get that
	\begin{align*}
		&\sum_{j=0}^{k \vee n-k} (-1)^j \sum_{\rho \in P_{1, 2}^{j,k}(n)} q^{\iota'(\rho)} \prod_{\{\ell_1, \ell_2\} \in \rho} \langle h_{\ell_1}, h_{\ell_2} \rangle \sum_{j' = 0}^{k-j \vee n-k-j}\sum_{\sigma \in P_{1, 2}^{j', k-j}} q^{\iota(\sigma)} \prod_{\{\ell_1, \ell_2\} \in \sigma} \langle h_{\ell_1}, h_{\ell_2} \rangle W( (\xi_{\rho})_{ \sigma})\\
		&= \sum_{j=0}^{k \vee n-k} (-1)^j \sum_{j' = 0}^{k-j \vee n-k-j} \sum_{\rho \in P_{1, 2}^{j, k}(n)} \sum_{\sigma \in P_{1, 2}^{j', k-j}(n-2j)} q^{\iota'(\rho) + \iota(\sigma)} \prod_{\{\ell_1, \ell_2\} \in \rho \cup \sigma} \langle h_{\ell_1}, h_{\ell_2} \rangle W(\xi_{\rho \cup \sigma})\\
		&= \sum_{j=0}^{k \vee n-k} (-1)^j \sum_{m = j}^{k \vee n-k} \sum_{\rho \in P_{1,2}^{j, k}(n)} \sum_{\sigma \in P_{1, 2}^{m-j, k-j}(n-2j)} q^{\iota'(\rho) + \iota(\sigma)} \prod_{\{\ell_1, \ell_2\} \in \rho \cup \sigma} \langle h_{\ell_1},  h_{\ell_2} \rangle W(\xi_{\rho \cup \sigma})\\
		&= \sum_{m = 0}^{k \vee n-k} \sum_{j= 0}^m (-1)^j \sum_{\rho \in P_{1, 2}^{j, k}(n)} \sum_{\sigma \in P_{1, 2}^{m-j, k-j}(n-2j)} q^{\iota'(\rho) + \iota(\sigma)} \prod_{\{\ell_1, \ell_2\} \in \rho \cup \sigma} \langle h_{\ell_1}, h_{\ell_2} \rangle W(\xi_{\rho \cup \sigma})\\
		&= \sum_{m=0}^{k \vee n-k} \sum_{\pi \in P_{1,2}^{m, k}(n)} \sum_{j = 0}^m (-1)^j \sum_{\substack{\rho \in P_{1, 2}^{j, k}(n)\\ \sigma \in P_{1,2}^{m-j,k}(n-2j)\\ \rho \cup \sigma = \pi}} q^{\iota'(\rho) + \iota'(\sigma)} \prod_{\{\ell_1,\ell_2\} \in \pi}\langle h_{\ell_1}, h_{\ell_2} \rangle W(\xi_\pi),
	\end{align*}
	where $\rho \cup \sigma \in P_{1, 2}^k(n)$ has pairs from both $\rho$ and $\sigma$. From here, it is clear that for $m =0$, we simply have $W(\xi)$. Therefore, the following claim finishes the proof. 
	\begin{claim}
		For $m \geq 1$,
		\[
		\sum_{j = 0}^m (-1)^j \sum_{\substack{\rho \in P_{1, 2}^{j, k}(n)\\ \sigma \in P_{1,2}^{m-j,k}(n-2j)\\ \rho \cup \sigma = \pi}} q^{\iota'(\rho) + \iota'(\sigma)}  = 0
		\]
	\end{claim}

	\begin{proof}[Proof of Claim]
		We proceed by induction. For $m=1$, fix $\pi \in P_{1,2}^{1, k}(n)$. We have
		\[
		q^{\iota'(\pi)} - q^{\iota(\pi)} = 0
		\]
		since $\iota'(\pi) = \iota(\pi)$ for $\pi \in P_{1,2}^{1,k}(n)$. Let $\pi_m \in P_{1, 2}^{m, k}(n)$, and let $\pi_{m-1} \in P_{1,2}^{m-1, k}(n)$ be such that $\pi_m \setminus \pi_{m-1} = \{\ell_1, \ell_2\}$ where $\ell_1 < \ell_1'$ for all pairs $\{\ell_1', \ell_2'\} \in \pi_{m-1}$. By our inductive hypothesis, we have that $S_{\pi_{m-1}} = 0$. However, we have
		\[
		S_{\pi_m} = q^{c_1} S_{\pi_{m-1}} + q^{c_2} S_{\pi_{m-1}} = 0
		\]
		where $c_1 = \iota(\pi_m) - \iota(\pi_{m-1})$ and $c_2 = \iota'(\pi_m) - \iota'(\pi_{m-1})$. The first term comes from $\{\ell_1, \ell_2\} \in \sigma$, and the second term similarly comes from $\{\ell_1, \ell_2\}\in \rho$. This finishes the proof of the claim and the proposition. 
	\end{proof}
	
	\renewcommand{\qedsymbol}{}
	\end{proof}
	\begin{proof}[of Proposition \ref{coordbound}]
		Since $\|w_{n,k}^j\|_{cb} \leq C_q$ by Lemma \ref{little2}, we get from Proposition \ref{big} that 
		\[
		\|\beta_{n,k}\|_{cb} \leq \sum_{j = 0}^{k \vee n-k} |q|^{ {j \choose 2}} \|w_{n,k}^j\|_{cb} < C_q
		\]
	\end{proof}

We are now ready to prove Theorem A.
	\begin{thma}
	For all $-1 < q < 1$ and all $\dim(H) \geq 2$, 
	\begin{enumerate}
			\item $\Gamma_q(H)$ has the weak* completely contractive approximation property.
	\item $\mathcal{A}_q(H)$ has the completely contractive approximation property.

	\end{enumerate}
	\label{CCAP}
\end{thma}

	\begin{proof}
	For (1), we follow Haagerup's standard argument from \cite{haagerup79} except the cb-norm of the projections onto words of length $n$ are bounded by $cn^2$ instead of $cn$. 
	From second quantization, we know that $E = \Gamma_q(P)$ for any projection $P$ is a conditional expectation and $T_t = \Gamma_q(e^{-t}Id)$ is a ucp semigroup and thus $\|E\|_{cb} = \|T_t\|_{cb} = 1$. Furthermore $\|T_t|_{F_n}\| = e^{-nt}$ where $F_n$ is the subspace spanned by the Wick words of degree $n$. We now estimate the cb-norm for $P_n: \Gamma_q(H) \to F_n$.
	From above we know that \[
	\|P_n\|_{cb} \leq \|\Phi_n\|_{cb} \|\Psi_n\|_{cb} \leq C_q n^2.
	\]
	Let $P_{\leq N} = \sum_{n=0}^N P_n$. Finally, we define the following net of maps.
	\[
	U_{\alpha} = T_{t_\alpha} \circ P_{\leq N_\alpha} \circ E_{k_\alpha}
	\]
	Where $E_{k_\alpha} = \Gamma_q(F_{k_\alpha})$ where $F_{k}$ is a sequence of projections of rank $k$ whose union is the identity, and $t_\alpha$ and $N_{\alpha}$ satisfy 
	\[
	\sum_{n= N_\alpha + 1}^\infty C_q e^{-nt} n^2 <  \varepsilon
	\]
	for all $\alpha$, but $n_\alpha \to \infty$ and $t_\alpha \to 0$. Clearly $U_\alpha$ is finite rank, completely bounded with norm less than $1 + \varepsilon$, and converges to the identity in the point-weak* topology. 
	
	For (2), observe that what we have shown is that there exists functions $f_\alpha: \mathbb{N} \to \mathbb{R}$ where $\alpha = (t, m, \varepsilon)$ such that 
	\[
	\sum_{n =N+1}^\infty e^{-t n}n^2 \leq \varepsilon
	\]
	which satisfies the following conditions:
	\begin{enumerate}
		\item The pointwise limit of $f_\alpha$ is $1$.
		\item $f_\alpha$ has finite support for each $n$. 
		\item $\|f_\alpha(N)\|_{cb} \leq 1 + \varepsilon$.
	\end{enumerate}
	Recall that $N$ is the number operator, which generates the semigroup $T_t$. Since $\Gamma_q(H)$ is faithfully represented on $\mathcal{F}_q(H)$, we have that the Wick words linearly generate $\mathcal{A}_q(H)$. For $\xi \in H^{\otimes n}$, we have that
	\[
	f_\alpha(N)W(\xi) = f_\alpha(n) W(\xi).
	\]
	Therefore, $f_\alpha(N)$ converges to the identity in the point-norm topology and $\lim_\alpha \|f_\alpha\|_{cb} = 1$  for appropriately chosen $\alpha$. 
\end{proof}
	\begin{rem}
		We observe that $T_t: \Gamma_q(H) \to \mathcal{A}_q(H)$ since for $x \in \Gamma_q(H)$, we have that
		\[
		T_t(x) = \sum_{n \geq 0} e^{-nt} P_n(x).
		\]
		Clearly $P_n(x) \in \mathcal{A}_q(H)$ for $H$ finite dimensional since the range of $P_n$ is spanned by Wick words. For $H$ infinite dimensional, we get that $P_n(x) \in \mathcal{A}_0(H)$ for all $n$ and $x \in \Gamma_0(H)$. Recall from \cite{nou} that $\|W(\xi)\|_\infty \leq C_q \|W(x)\|_2$. Therefore for all $\varepsilon > 0$,
		\[
		\|T_t(x) - \sum_{n = 0}^M e^{-nt}P_n(x)\|_{} \leq \|\sum_{n> M+1} e^{-nt}P_n(x)\| < \varepsilon \|x\|
		\]
		for $M$ such that $C\sum_{n > M}e^{-nt}n^2 < \varepsilon$. Therefore, $F_M(x) := \sum_{n = 0}^n e^{-nt}P_n(x)$ converges in norm to $T_t(x)$. Hence $T_t(x) \in \mathcal{A}_q(H)$. It is obvious that $\mathcal{A}_q(H)$ has the CCAP since we may then simply apply the projections $P_n$ to an element of the form $T_t(x)$. 
	\end{rem}

\section{Weak Containment}
In this section, we shall show that while $L_{0}^{2}(\Gamma_q(H \oplus H))$ is not obviously weakly contained in the coarse bimodule, there is a subbimodule of $L_{0}^{2}(\Gamma_q(H \oplus H))$ which is weakly contained in the coarse bimodule. Define the following subspaces of $L_{0}^{2}(\Gamma_q(H \oplus H))$.
\begin{align*}
	F_m &= \span\left\{ W(h_1 \otimes  \ldots \otimes h_n)| \exists \iota_1 \ldots \iota_m \in \{1, \ldots, n\}, h_{\iota_k} \in 0 \oplus H \right\}^{\|\cdot \|_{2}}\\
	E_m &= \oplus_{k = 0}^m F_k
\end{align*}
Note that $F_m$ and $E_m$ are $\Gamma_q(H)$-$\Gamma_q(H)$-bimodules simply by the action restricted from $L_{0}^{2}(\Gamma_q(H \oplus H))$. The main result of this section is the following.
\begin{prop}
	Let $m > -\frac{\log(d)}{2 \log(|q|)}$ where $d = \dim(H)$. Then $E_{m-1}^\perp \prec L_{}^{2}(\Gamma_q(H)) \bar{\otimes} L_{}^{2}(\Gamma_q(H))$. 
	\label{weakcont}
\end{prop}

It will turn out that this sub-bimodule, $E_{m-1}^\perp$, will be ``large enough'' to replace $L_{0}^{2}(\Gamma_q(H \oplus H))$ in the proof of strong solidity from \cite{HS} and \cite{OP1}. Throughout this section, we shall denote $\Gamma_q(H)$ by $\mathcal{M}$, $\Gamma_q(H \oplus H)$ by $\widetilde{\mathcal{M}}$, $(h, 0) \in H \oplus H$ by simply $h$, and $(0, h)$ by $\tilde{h}$. Define $\Phi_{\xi, \eta}: L_{}^{p}(\mathcal{M}) \to L_{}^{p}(\mathcal{M})$ by $\Phi_{\xi, \eta}(x) = E_\mathcal{M}(W(\xi)^\ast x W(\eta))$ for $\xi, \eta \in F_k$.

\begin{proof}[Idea of the proof of Proposition \ref{weakcont}]
	Let $\tilde{h}, \tilde{k} \in 0 \oplus H$ and define 
	\[
	\Phi_{h,k}(x) = E_\mathcal{M}(s(\tilde{h})x s(\tilde{k}))
	\]
	as an operator on $L_{}^{2}(\mathcal{M})$. It is straightforward to see that 
	\[
	\Phi_{h,k}(x) = \sum_{n=0}^\infty q^n P_n(x) \langle h, k \rangle 	
	\]
	where $P_n$ is the projection onto Wick products of order $n$ as in the previous section. This map is the same as the map $\Xi_q$ defined in \cite{shlyakh04} up to the factor $\langle h, k \rangle $. Since $\Phi_{h,k}$ is a diagonal operator, its Schatten p-norm can be easily computed by
	\[
	\|\Phi_{h,k}\|_{S_p} = \left( \sum_{n=0}^\infty |q|^{pn} d^n \right)^{\frac{1}{p}}.
	\]
	This sum clearly converges if and only if $|q|^p d < 1$, so $\Phi_{h,k}$ is Schatten $p$-class if and only if $p > -\frac{\log(d)}{\log(|q|)}$. By Lemma \ref{weakcontHS}, we can only prove that a bimodule is weakly contained if $T_\xi$ is Hilbert-Schmidt. With this in mind, we observe that since $\Phi_{h,k}$ is Schatten p-class for $p > - \frac{\log(d)}{\log(|q|)}$, $\Phi_{h,k}^m$ is Hilbert-Schmidt for $m > - \frac{\log(d)}{2\log(|q|)}$. For  $\xi, \eta \in E_{m-1}^\perp$, 
	\[
	\Phi_{\xi, \eta}(x) := E_\mathcal{M}(W(\xi)^\ast x W(\eta))
	\]
	is an $N$-fold product of $\Phi_{h_j,k_i}$ as well as some bounded operators for $N > m$ and thus is Hilbert-Schmidt. The result then follows from Lemma \ref{weakcontHS}.
\end{proof}

\begin{lemma}
	If $\xi, \eta \in F_k$, then $\Phi_{\xi, \eta}: L^2(\mathcal{M}) \to L^2(\mathcal{M})$ is Schatten $p$-class for $p > -\frac{\log(d)}{k\log(|q|)}$. In particular $\Phi_{\xi, \eta}$ is Hilbert-Schmidt for $k \geq -\frac{\log(d)}{2 \log(|q|)}$. 	\label{tens}
\end{lemma}
We shall first need two additional lemmas.
\begin{lemma}
	Let $H = \oplus_{j \geq 0} H_j$ be a graded Hilbert space and $A= [A_{ij}]: H \to H$ be an operator such that
	\begin{enumerate}
		\item $A_{ij} = 0$ if $|i -j| \geq L$ for some $L > 0$.
		\item There exists $j_0$ such that 
			\[
			\|A_{ij}\| \leq C r^{jk}
			\]
			for all $j \geq j_0$ and for some constants $0 < r < 1$, $k$ and $C$ independent of $i$ and $j$.
		\item $\dim(H_j) = d^{j}$.
	\end{enumerate}
	Then $A \in S_p(H)$ for $p > -\frac{\log(d)}{k \log(r)}$.
	\label{tech}
\end{lemma}

\begin{proof}
	Let $K_1 = \oplus_{j = 0}^{j_0 -1} H_j$ and $K_2 = \oplus_{j \geq j_0}H_j$, then 
	\begin{align*}	\|A\|_{S_p} \leq \|A: K_1 \to K_1\|_{S_p} &+ \|A: K_1 \to K_2\|_{S_p}\\
		&+ \|A^\ast: K_1 \to K_2\|_{S_p} + \|A: K_2 \to K_2\|_{S_p}\end{align*}
		Since $K_1$ is finite dimensional, we may control the first three norms simply by a constant depending on the dimension of $K_1$ and the norm of $A$, so we only must estimate 
		\[
		\|A: K_2 \to K_2\|_{S_p}.
		\]
		For $\|A: K_2 \to K_2\|_{S_p}$, we have that
\begin{align*}
	\|A\|_{S_p} &\leq C + \sum_{\ell = -L}^L \sum_{j \geq L} \|A_{j, j+\ell}\|_{S_p}\\
	&\leq C + \sum_{\ell = -L}^L((\sum_{j = j_0}^\infty  d^j \|A_{j+\ell, j}\|^p)^\frac{1}{p})\\
	&\leq C + (\sum_{j = j_0}^\infty (2L)d^j C^p r^{jkp})^\frac{1}{p}
	\end{align*}
	The sum converges if and only if $d r^{kp} < 1$, which is equivalent to $p > -\frac{\log(d)}{k \log(r)}$.
\end{proof}

\begin{lemma}
	Suppose $\xi \in (H \oplus H)^{\otimes_{} n_1} \cap F_k$, $\eta \in (H \oplus H)^{\otimes_{} n_2} \cap F_k$, and that $\zeta_1 \in H^{\otimes j}$, and $\zeta_2 \in H^{\otimes_{}i}$ for $j \geq 2(n_1+n_2)$ and $|i-j| \leq n_1 + n_2$.Then we have that

	\[
	|\langle W(\zeta_2), \Phi_{\xi,\eta}(W(\zeta)) \rangle| \leq C_{\xi, \eta} |q|^{kj} \|\zeta_1\|_{2} \|\zeta_2\|_{2}.
	\]
	\label{tech2}
\end{lemma}

\begin{proof}
	First observe that 
	\begin{align*}
		\langle W(\zeta_2), \Phi_{\xi,\eta}(W(\zeta_1)) \rangle &= \tau\left( W(\zeta_2)^\ast E_\mathcal{M}(W(\xi)^\ast W(\zeta_1) W(\eta) \right)\\
		&= \tau\left( W(\zeta_2)^\ast W(\xi)^\ast W(\zeta_1) W(\eta) \right)\\
		&= \langle W(\xi)W(\zeta_2) , W(\zeta_1)W(\eta) \rangle 
	\end{align*}
	Applying our formula from Lemma \ref{threeprod}, we get that
	\[
	W(\xi)W(\zeta_2) = \sum_{\nu \in P_{1,2}(m_1)} q^{\iota(\nu)} W(\xi \otimes_{} \zeta_2)_{\nu} \prod_{\nu} (\xi \otimes_{} \zeta_2)_p
	\]
	where $P_{1,2}$ and $\iota: P_{1,2} \to \mathbb{R}$ are defined in Section 3, $m_1 = n_1 + i$, $(\xi \otimes_{} \zeta_1)_\nu$ denotes $\xi \otimes_{} \zeta_2$ with the pairs of $\nu$ removed, and 
	\[
	\prod_\nu (\xi \otimes_{} \zeta_2)_p  = \prod_{ \left\{ \alpha, \beta \right\} \in \nu} \langle h_\alpha,  f_\beta \rangle.
	\]
	Note that the $k$ tensors in $\xi$ which belong to the subspace $0 \oplus H$ cannot be an element of a pair since they are orthogonal to each of the tensors of $\zeta_2$. Naturally, we have a similar formula for $W(\zeta_1)W(\eta)$. We now need to prove two claims.
	\begin{claim}
		\[
		\sum_{\nu \in P_{1,2}(m_1)} |q|^{\iota(\nu)} \|(\xi \otimes_{} \zeta_2\|_{2} |\prod_\nu (\xi \otimes_{} \zeta_2)_p| \leq C_{q, n_1} \|\xi\|_{2} \|\zeta_2\|_{2} 
		\]
	\end{claim}

	\begin{proof}
		Following from Lemma 1 in \cite{nou}, we have
	\begin{align*}
		\sum_{\nu \in P_{1,2}(m_1)} |q|^{\iota(\nu)} \|(\xi \otimes_{} \zeta_2)_\nu\|_{2} &|\prod_\nu (\xi \otimes_{} \zeta_2)_p| \\
		&\leq C_q \sum_{\nu \in P_{1,2}(m_1)} |q|^{\iota(\nu)} \|\xi_\nu\|_{2} \|(\zeta_2)_\nu\|_{2} \prod_{ \left\{ \alpha, \beta \right\} \in \nu} \|\xi_\alpha\|_H \|\zeta_{2, \beta}\|_{H}\\
		&\leq C_{q, n_1}\sum_{\nu \in P_{1,2}(m_1)}  |q|^{\iota(\nu)} \|\xi\|_{2} \|\zeta_2\|_{2}\\
		&\leq C_{q,n_1} \|\xi\|_{2} \|\zeta_2\|_{2}
	\end{align*}
	\end{proof}
	We clearly have the same estimate for $W(\zeta_1)W(\eta)$. 
	\begin{claim}
		For any fix $\nu_1 \in P_{1,2}(m_1)$ and $\nu_2 \in P_{1,2,}(m_2)$, we have
		\[
		|\langle (\xi \otimes_{}\zeta_2)_{\nu_1}, (\zeta_1 \otimes_{} \eta)_{\nu_2} \rangle | \leq C_q |q|^{jk} \|(\xi \otimes_{} \zeta_2)_{\nu_1}\|_{2}\|(\zeta_1 \otimes_{} \eta)_{\nu_2}\|_{2}
		\]
	\end{claim}

	\begin{proof}
		Recall that both $(\xi \otimes_{}\zeta_2)_{\nu_1}$ and $(\zeta_1 \otimes_{} \eta)_{\nu_2}$ contain $k$ tensors from $0 \oplus H$ as these elements cannot be contained in the pairs of $\nu_1$ and $\nu_2$. The elements must be in the same eigenspace of the number operator, thus $n_1 + i - 2P(\nu_1) = n_2 + j - 2P(\nu_2) := m$, where $P(\nu)$ denotes the number of pairs of an element of $P_{1,2}(m_j)$. By our assumptions above, there must be more than $j - n_2$ tensors which are elements of $H \oplus 0$ to the left of the left-most tensor which is contained in $0 \oplus H$ in $(\zeta_1 \otimes_{} \eta)$ (since $\eta$ contains all of the tensors which are elements of $0 \oplus H$). Similarly, there must be more than $i - n_1$ tensors to the right of the right-most element of $0 \oplus H$. Therefore, using quasi-multiplicativity of the function $q^{\iota}$ and Cauchy-Schwarz, we have
		\begin{align*}
			|\langle (\xi \otimes_{} \zeta_2)_{\nu_1}, (\zeta_1 \otimes_{} \eta)_{\nu_2} \rangle| &\leq \sum_{\sigma \in S_m} q^{\iota(\sigma)} \prod_j |\langle h_j, f_{\sigma(j)} \rangle | \\
			&= |q|^M \sum_{\sigma \in S_k} \sum_{\sigma' \in S_{m-k}} |q|^{\iota(\sigma) + \iota(\sigma')} \prod_j \prod_{j} |\langle h_j, f_{\sigma(j)} \rangle|\\
			&\leq |q|^M \|(\xi \otimes_{} \zeta_2)_{\nu_1}\|_{2} \|(\zeta_1 \otimes_{} \eta)_{\nu_2}\|_{2}
		\end{align*}
		where $M$ is the element with the minimal number of inversions which allows us to match all of the tensors which are elements of $0 \oplus H$ on the left with the corresponding tensors on the right. By what we have observed above, we have that 
		\[
		M \geq (j - n_2 + i - n_1)k \geq (2j - 2n_1 -2n_2)k \geq jk
		\]
	\end{proof}
	Now we can easily prove the lemma. We see that
	\begin{align*}
		&|\langle W(\xi)W(\zeta_2), W(\zeta_1)W(\eta) \rangle| \\
		&\leq \sum_{\substack{\nu_1 \in P_{1,2}(m_1)\\ \nu_2 \in P_{1,2}(m_2)}} |q|^{\iota(\nu_1) + \iota(\nu_2)} |\langle (\xi \otimes_{}\zeta_2)_{\nu_1}, (\zeta_1 \otimes_{} \eta)_{\nu_2} \rangle| |\prod_{\nu_1} (\xi \otimes_{} \zeta_2)_p \prod_{\nu_2} (\zeta_1 \otimes_{} \eta)_p|\\
		&\leq |q|^{jk} \sum_{\substack{\nu_1 \in P_{1,2}(m_1)\\ \nu_2 \in P_{1,2}(m_2)}} |q|^{\iota(\nu_1)+ \iota(\nu_2)} \|(\xi \otimes_{} \zeta_2)_{\nu_1}\|_{2} \|(\zeta_1 \otimes_{} \eta)_{\nu_2}\|_{2} |\prod_{\nu_1}(\xi \otimes_{} \zeta_2)_p \prod_{\nu_2} (\zeta_1 \otimes_{} \eta)_p|\\
		&\leq C_{q, n_1, n_2} |q|^{jk} \|\zeta_1\|_{2} \|\zeta_2\|_{2} \|\xi\|_{2} \|\eta\|_{2}
	\end{align*}
\end{proof}

\begin{proof}[Proof of Lemma \ref{tens}]
	The previous lemma immediately implies that $\|(\Phi_{\xi, \eta})_{ij}\|_{} \leq C_{q, \xi, \eta} |q|^{jk}$. Let $\left\{ e_{\iota} \right\}$ and $\left\{ f_{\iota} \right\}$ be orthonormal bases of $H^{\otimes_{}i}$ and $H^{\otimes_{}j}$ respectively and $\zeta_1 = \sum_{\iota} b_\iota f_\iota$ and $\zeta_2 = \sum_{\iota} a_\iota e_\iota$. Then we have
	\begin{align*}
		|\langle \zeta_2, \Phi_{\xi, \eta}(\zeta_1) \rangle | &\leq |\langle \sum_{\iota_1} a_\iota e_\iota, \Phi_{\xi, \eta}(\sum_{\iota_2}b_\iota f_\iota) \rangle | \\
		&\leq \sum_{\iota_1, \iota_2}|a_\iota| |b_\iota| |\langle e_\iota,\Phi_{\xi, \eta}(f_\iota)  \rangle |\\
		&\leq C_{q, \xi, \eta} |q|^{jk} \sum_{\iota_1, \iota_2} |a_\iota| |b_\iota| \leq C_{q, \xi, \eta} |q|^{jk} \|\zeta_1\|_2 \|\zeta_2\|_{2}
	\end{align*}
	Therefore, by Lemma \ref{tech}, we have that $\Phi_{\xi, \eta} \in S_p$ for $p > -\frac{\log(d)}{k \log(|q|)}$. 
	\end{proof}


\begin{proof}[Proof of Proposition \ref{weakcont}]
	Let $\xi \in F_{k}$ for $k \geq -\frac{\log(d)}{2 \log(|q|)}$. With the natural $\mathcal{M}$-$\mathcal{M}$-bimodule structure on $F_k$, we can see that
	\[
	\langle \xi, x \xi y \rangle  = \tau(W(\xi)^\ast x W(\xi)y) = \tau(E_\mathcal{M}(W(\xi)^\ast x W(\xi)y)) = \tau(\Phi_{\xi, \xi}(x)y)
	\]
	for $x, y \in \mathcal{M}$. Note that $\Phi_{\xi, \xi}$ coincides with $T_{\varphi_\xi}$ in Lemma \ref{weakcontHS}. Thus since $\Phi_{\xi, \xi}$ is Hilbert-Schmidt for all $\xi$, $F_k \prec L_{}^{2}(\mathcal{M}) \otimes L_{}^{2}(\mathcal{M})$ by Lemma \ref{weakcontHS}. Therefore, $E_{m-1}^{\perp} = \oplus_{k \geq m} F_k \prec \oplus_{k \geq m} L_{}^{2}(\mathcal{M}) \otimes_{} L_{}^{2}(\mathcal{M})$. Since a direct sum of coarse bimodules is weakly contained in the coarse bimodule, we have proved the proposition.
\end{proof}

\section{Strong Solidity}
As shown in the previous section, $L_{0}^{2}(\Gamma_q(H \oplus H))$ is not necessarily weakly contained in the coarse correspondence for $q^2 \dim(H) \geq 1$. However, the submodule $E_k^\perp$ is weakly contained in the coarse bimodule for sufficiently large $k$. This requires us to modify Popa's s-malleable deformation (\cite{popa08} Lemma 2.1) estimate slightly to suit our new situation. What we need to know is that the image of $\Gamma_q(H)$ under the automorphism group $\alpha_t$ has a ``large enough'' intersection with $E_k^\perp$. This is the purpose of the following proposition.

\begin{prop}
	\label{bigprop}
	For a fixed $k \geq 1$, there exists a constant depending only on $k$, $C_k$ such that

	\[	\|(\alpha_{t^k} - id)(x)\|_2 \leq C_k \|E_{k - 1}^\perp \alpha_t(x)\|_2	\]
	for  $x \in \oplus_{m \geq k} H^{\otimes m} \subset \mathcal{F}_q(H)$ and $t < 2^{-k}$.
\end{prop}

\begin{proof}
	
	Let $x = h_1 \otimes\cdots \otimes h_n$, $y = k_1 \otimes \cdots \otimes k_n$. Note that $\alpha_{t^k} - id$ and $E_{k-1}^\perp \alpha_t$ are both tensor length-preserving operators, so it suffices to prove this estimate on $H^{\otimes n}$ for $n \geq k$.
      We calculate	
      \begin{align*}
	      \langle E_{k-1}^\perp \alpha_t(x), E_{k-1}^\perp \alpha_t(y) \rangle &= \sum_{m = k}^n \langle F_m \alpha_t(x), F_m \alpha_t(y) \rangle \\
	      &= \sum_{m = k}^n \sum_{A, B \subset \{1, \ldots, n\}} e^{-2t(n-m)}(1 - e^{-2t})^{m}\langle x_{A^c} \otimes x_A, y_{B^c} \otimes y_B \rangle_q
      \end{align*}
      where $x_{A^c} \otimes x_A$ denotes that the indices belonging to $A^c$ come from $H \oplus 0$ and the indices belonging to $A$ come from $0 \oplus H$. Now we expand the $q$-inner product to get 
      \begin{align*}
	      &\sum_{m = k}^n \sum_{A, B \subset\{1, \ldots, n\}} e^{-2t(n-m)}(1 - e^{-2t})^m \langle x_{A^c} \otimes x_A, y_{B^c} \otimes y_B \rangle\\  &= \sum_{m =  k}^n e^{-2t(n-m)} (1 - e^{-2t})^{m} \sum_{A} \sum_{x \in S_n/S_A \times S_{A^c}} \sum_{\sigma \in x} q^{\iota(\sigma)} \prod_{j} \langle h_j, k_{\sigma(j)}\rangle
      \end{align*}
      since for each $A, B \subset \{1, \ldots, n\}$ only those permutations $\sigma \in S_n$ which map $A$ to $B$ contribute to the inner product. This is equivalent to summing over the right cosets in $S_n/S_A \times S_{A^c}$. However, since we are summing over all permutations in all of the cosets, for each fixed $A$, we are summing over all the permutations, and so we just get $\langle x, y \rangle_q$ for each fixed $A$. Therefore we get that 
      \[
      \sum_{m = k} \sum_{|A| = k} e^{-2t(n-m)} (1- e^{-2t})^m \langle x, y \rangle_q = \sum_{m = k}^n e^{-2t(n-m)}(1 - e^{-2t})^{m} {n \choose m} \langle x, y \rangle_q.
      \]
      For $\alpha_{t^k} - id$ we have
      \begin{align*}
	      \langle (\alpha_{t^k} - id)(x), (\alpha_{t^k} - id)(y) \rangle &= \langle \alpha_{t^k}(x), \alpha_{t^k}(y) \rangle - \langle \alpha_{t^k}(x), y \rangle  - \langle x, \alpha_{t^k}(y) \rangle + \langle x, y \rangle \\
	      &= 2(\langle x, y \rangle  - \langle x, T_{t^k}(y) \rangle )\\
	      &= 2(1 - e^{-nt^k})\langle x, y \rangle .
      \end{align*}

      Therefore, we only have to show that $2(1-e^{-nt^k}) < C \sum_{m = k}^n {n \choose m} e^{-2(n-m)t}(1 - e^{-2t})^{m}$ for some $C$ independent of $n$ and $t$. We choose $M_k$ such that $e^{-2nt} \sum_{m=0}^{k-1} C_m n^m t^m < \frac{1}{2}$ for $nt > M_k$. Suppose $nt < M_k$. We have
      \[
      \sum_{m = k}^n {n \choose m} e^{-2(n-m)t}(1 - e^{-2t})^{m} > {n \choose k} e^{-2(n-k)t}(1 - e^{-2t})^{k} > C_k n^k t^{k},
      \]
      and
      \[
      2(1 - e^{-nt^k}) < C_k nt^k < C_k n^k t^k
      \]
      Now suppose that $nt > M_k$. Then $n > 2^k$ since $t < 2^{-k}$ and so,
      
      \begin{align*}
	      \sum_{m = k}^n {n \choose m} e^{-2(n- m)t}(1 - e^{-2t})^m &= 1 - \sum_{m = 0}^{k-1} {n \choose m}e^{-2(n-m)t}(1 - e^{-2t})^m\\
	      &= 1 - e^{-2nt}\sum_{m=0}^{k-1} {n \choose m} e^{2mt}(1- e^{-2t})^m\\
	      &\geq 1 - e^{-2nt} \sum_{m=0}^{k-1}C_m n^m t^m
	      \geq \frac{1}{2}
            \end{align*}
	    However, clearly $2(1-e^{-nt^k}) < 2$ and so we have proved the statement for all $n$ and $t < 2^{-k}$.
	\end{proof}

	Now we may prove Theorem B, following the proof of Theorem 3.5 in \cite{HS}. There are a number of modifications since we are using a proper sub-bimodule of $L_{0}^{2}(\Gamma_q(H \oplus H))$.  

\begin{thmb}
	For all $-1 < q < 1$ and all $\dim(H) < \infty$, $\Gamma_q(H)$ is strongly solid.
\end{thmb}

\begin{proof}
	Let $P \subset \Gamma_q(H)$ be a diffuse, amenable subalgebra. We want to prove that $\mathcal{N}_{\Gamma_q(H)}(P)''$ is also amenable. $P$ is not rigid with respect to the deformation $\alpha_t$ (Lemma \ref{rigid}), and $P$ is weakly compact inside of $\Gamma_q(H)$.  Since $P \subset \Gamma_q(H)$ is weakly compact, there is a net of elements $(\eta_n) \in L^2(P \otimes \bar{P})$ which satisfy 
	\begin{enumerate}
		\item $\lim_n \|\eta_n - (v \otimes \bar{v}) \eta_n\|_2 = 0$, $\forall v \in \mathcal{U}(P)$,\\
		\item $\lim_n \|\eta_n - Ad(u \otimes \bar{u})\eta_n\|_2 = 0$, $\forall u \in \mathcal{N}_{\Gamma_q(H)}(P)$ \mbox{ and,}\\
		\item $\langle (1 \otimes \bar{x}) \eta_n, \eta_n \rangle = \tau(x) = \langle \eta_n, (x \otimes 1) \eta_n \rangle $.
	\end{enumerate}
	Following \cite{HS}, let $\mathcal{G}$ denote $\mathcal{N}_{\Gamma_q(H)}(P)$, and let $z \in \mathcal{Z}(\mathcal{G}' \cap \Gamma_q(H))$ be a non-zero projection. Since $\alpha_t$ does not converge uniformly on $(P)_1$, $\alpha_t$ does not converge uniformly on $(Pz)_1$ and so $\alpha_t$ does not converge uniformly on $\mathcal{U}(Pz)$ either. Therefore there exist $0 < c < 1$, a sequence $(u_k) \in \mathcal{U}(Pz)$, and a sequence $t_k \to 0$ such that $\| \alpha_{t_k}(u_k z) - (E_{m-1} \circ \alpha_{t_k})(u_k z)\|_2 \geq c \|z\|_2$  $\forall k \in \mathbb{N}$, by Proposition \ref{bigprop}. Since $\|\alpha_{t_k}(u_k z)\|_2 = \|z\|_2$, we get 
	\begin{equation}	
		\label{star}
		\| (E_{m-1}^\perp \circ \alpha_{t_k})(u_k z)\|_2 \leq \sqrt{1 - c^2} \|z\|_2
		\end{equation}
	for all $k \in \mathbb{N}$.
	Let $P_\mathcal{H} = E_{m-1}^\perp$. Define for all $n$ and $k$
	\begin{align*}
		\eta_n^k &= (\alpha_{t_k} \otimes 1)(\eta_n) \in L_{}^{2}(\Gamma_q(H \oplus H)) \bar{\otimes} L_{}^{2}(\bar{\Gamma_q(H)})\\ 
		\xi_n^k &= (P_\mathcal{H}^\bot \alpha_{t_k} \otimes 1)(\eta_n) \in (L_{}^{2}(\Gamma_q(H \oplus H)) \ominus \mathcal{H}) \bar{\otimes} L_{}^{2}(\bar{\Gamma_q(H)})\\
		\zeta_n^k &= (P_\mathcal{H} \alpha_{t_k} \otimes 1)(\eta_n) \in \mathcal{H} \bar{\otimes} L_{}^{2}(\bar{\Gamma_q(H)})
	\end{align*}
	Observe that 
	\begin{align*}
		\|(x \otimes_{} 1) \eta_n^k \|_2^2 &= \langle (x \otimes_{} 1)(\alpha_{t_k} \otimes 1) \eta_n, (x \otimes_{} 1)(\alpha_{t_k} \otimes 1)\eta_n \rangle 
	\end{align*}
	Also, for all $x \in \Gamma_q(H)$ we have
	\begin{align*}
		\|(x \otimes 1)\zeta_n^k\|_2 &= \|(x \otimes 1)(P_{\mathcal{H}} \otimes 1)\eta_n^k\|_2\\
		&= \|(P_\mathcal{H} \otimes 1)(x \otimes 1)\eta_n^k\|_2\\
		&\leq \|(x \otimes 1) \eta_n^k]\|_2\\
		&= \|x\|_2
	\end{align*}
	Therefore we have the following claim
	\begin{claim}
		For any $k$ sufficiently large, 
		\[
		\lim_n \|(z \otimes 1) \zeta_n^k\|_2 \geq \delta
		\]
	\end{claim}

	\begin{proof}
		Assume not. Following Houdayer-Shlyakhtenko, we get that this implies that
		\begin{align*}
			\lim_n \|(z \otimes 1) \eta_n^k - (E_{m-1} \alpha_{t_k}(u_k)z \otimes \bar{u_k})\xi_n^k\|_2 &\leq \delta
		\end{align*}
		However,
		\begin{align*}
			\|E_{m-1} \circ \alpha_{t_k}(u_k z)\|_2 &\geq \|E_{m-1} \circ \alpha_{t_k}(u_k)z\|_2 - \|z - \alpha_{t_k}(z)\|_2\\
			&\geq \lim_n \|(E_{m-1}\circ \alpha_{t_k}(u_k)z \otimes \bar{u_k})(\eta_n^k)\|_2 - \delta\\
			&\geq \lim_n \|(z \otimes 1)\eta_n^k\| - 2\delta\\
			&= \|z\|_2 - 2\delta \geq \sqrt{1 - c^2} \|z\|_2
		\end{align*}
		which contradicts \eqref{star}.
	\end{proof}
	From here, we may follow the remainder of the proof in \cite{HS} verbatim in what follows their Claim 3.6 since the bimodule $E_k^\perp$ is weakly contained in the coarse bimodule for sufficiently large $k$. 
	\end{proof}

\bibliography{refs}
\bibliographystyle{plain}

\end{document}